\documentclass{amsart}[11pt]
\usepackage{graphicx,url,amsthm,xspace,enumitem,subcaption}
\usepackage[dvipsnames]{xcolor}
\usepackage{amscd}
\usepackage{multirow}
\usepackage[utf8]{inputenc}
\usepackage{tikz}
\usetikzlibrary{braids}
\graphicspath{{images/}}

\usepackage{amssymb}
\usepackage{graphics}
\usepackage{graphicx}
\usepackage{amscd}
\usepackage{color}
\usepackage{amsmath, amsthm, epsfig,  psfrag}
\usepackage{graphicx,url,amsthm,xspace,enumitem}
\usepackage{pinlabel}
\usepackage{hyperref}
\usepackage{subcaption}
\usepackage[margin=3cm, marginpar=2.5cm]{geometry}

\usepackage[all,pdftex,arc,curve,color,frame]{xy}
\usepackage{xy}
\xyoption{all}
\usepackage[normalem]{ulem}

\usepackage{pinlabel}

\usepackage{import}

\newcommand{\C}{\mathbb{C}}
\newcommand{\R}{\mathbb{R}}

\newcommand{\Z}{\mathbb{Z}}

\newcommand{\cptwobar}{\overline{\C\textup{P}}\,\!^2}

\newcommand{\ti}{\tilde}
\renewcommand{\d}{\partial}
\renewcommand{\div}{\operatorname{div}}

\newtheorem{theorem}{Theorem}[section]
\newtheorem{lemma}[theorem]{Lemma}
\newtheorem{prop}[theorem]{Proposition}
\newtheorem{cor}[theorem]{Corollary}

\theoremstyle{definition}
\newtheorem{definition}[theorem]{Definition}
\newtheorem{remark}[theorem]{Remark}
\newtheorem{example}[theorem]{Example}

\numberwithin{equation}{section}

\title{Planar multilinks and rational singularities}

\author{M\'arton Beke}
\address{Alfréd Rényi Institute of Mathematics, Budapest\\ University of Technology and Economics, Budapest, Hungary}
\email{bekem@renyi.hu}

\author{Olga Plamenevskaya}
\address{Department of Mathematics, Stony Brook University, Stony Brook, NY,
11794,  U.S.A.}
\email{olga@math.stonybrook.edu}

\begin{document}

\begin{abstract}
Fibered multilinks are a generalization of classical fibered knots and open books
that arise in the study of surface singularities and Milnor fibrations. We prove that if 
the canonical contact structure on the link of a surface singularity is supported by a planar multilink open book, then the singularity must be rational, and that sandwiched singularities are characterized by admitting planar multilinks with a component of multiplicity 1.  We also show that some topological properties of planar open books extend to planar multilinks:  symplectic fillings are negative definite and cannot contain symplectic surfaces of positive genus, and
the image of the Heegaard Floer contact invariant vanishes in $HF_{red}$. Our results for singularities are based on these topological considerations, partly using Min--Roy--Wang's work on fillings of planar spinal open books, as well as the combinatorics of lattice embeddings. 
\end{abstract}

\maketitle

\section{Introduction} Open book decompositions play an important role in 3-dimensional contact topology. In this paper, we focus on links of normal surface singularities: in this context, an open book supporting the canonical contact structure on the link can be obtained from a Milnor-type fibration given by the germ of an analytic function on the singular surface. A number of generalizations of open books exist in the literature. Part of the open book data is a fiber bundle over the circle such that the fibers (surfaces with boundary, aka {\em pages}) meet along a {\em binding} (or its generalization). For a classical open book, the binding is an oriented fibered link, and the page meets the tubular neighborhood of each binding component along a longitude. The monodromy fixes the boundary of the page.  An example of this situation is the classical Milnor 
fibration~\cite{milnor1968singular} given by
$\arg f: S^3_{\epsilon} \setminus \{f(x, y)=0\} \to S^1$, where $f\in \C[x, y]$,
$f(0, 0)=0$ is a polynomial with a critical point at the origin in $\C^2$, locally reduced at $0$, and $S^3_{\epsilon}$ a sphere of small radius $\epsilon>0$ centered at $0$. 
If the polynomial is not locally reduced, then we get a generalized open book where each page meets the tubular neighborhood of a binding component along several longitudes of the knot, that is, $m\geq 1$~boundary components of the page come together along a binding component. (We will say that the multiplicity of this binding component is $m$.)  The monodromy permutes the boundary components of the page.    This type of fibered link is called a {\em multilink} in \cite{eisenbud1985three}. Open books with multilink binding are a special case of spinal open books studied in \cite{LVHMW, LVHMW2, HRW} and of rational open books~\cite{BEVHM}. However, the spinal class is a much broader generalization: the tubular neighborhood of the binding, given by the union of solid tori in the classical setting as well as for multilinks, is replaced by a {\em spine} given by a circle bundle over an arbitrary surface with boundary.  Min--Roy--Wang obtained strong results on symplectic fillings of planar spinal contact 3-manifolds, generalizing  Wendl's theorem \cite{We}. In~\cite{PS2, beke2025unexpected}, we used these constructions to compare symplectic fillings to Milnor fibers of links of certain singularities. In this paper, we study some related questions.

The relation between analytic properties of a singularity and the topology of its link is a classical topic. One can also consider contact-topological properties of the link equipped with its canonical contact structure. In \cite{ghiggini2021surface}, Ghiggini, Golla, and the second author showed that the contact 
link is supported by a planar open book if and only if the singularity is rational with reduced fundamental cycle, so the complexity of a surface singularity is reflected in the complexity of the open book decompositions of its link. Further, links of sandwiched singularities admit planar \textit{multilink} open books with a binding component of multiplicity one, \cite{PS2}. (For the necessary  definitions see Sections~\ref{sec:sandw} and~\ref{sec:spinal}.)   We prove the converse, giving a contact-topological characterization of the sandwiched class:  

\begin{theorem}\label{thm:equivalence}
	The link $(Y,\xi_{can})$ of a normal surface singularity admits a planar multilink open book  
	with a multiplicity 1 boundary component if and only if the singularity is sandwiched.
\end{theorem}

The multiplicity 1 hypothesis is essential. For example, $D_4$ is not sandwiched, but we show that its link is supported by a planar multilink open book with two boundary components, both of multiplicity 2. Dropping the multiplicity hypothesis, we have 

\begin{theorem}\label{thm:rational} If the link $(Y,\xi_{can})$ of a normal surface singularity is supported by a planar multilink, then the singularity is rational. 
\end{theorem}

The converse is not true: for example, $E_8$ is not supported by a planar multilink. We are able to give a complete characterization of the dual resolution graph of singularities whose links can be supported by planar multilinks (see Section~\ref{sec:kulikov}). They form a class of generalized Kulikov singularities of genus $0$ in the sense of Stevens \cite{stevens2018kulikov} (see also \cite[6.7.A]{nemethi2022normal}), although this description is perhaps less transparent than our result for the sandwiched case.

The starting point of our proofs is a number of topological  observations 
derived from the structural results of~\cite{HRW} that represent symplectic fillings of 
planar multilink open books as nearly Lefschetz fibrations. We show that several known properties of classical planar open books
 hold for 
planar multilinks, cf. \cite{etnyre2004planar, OSSPlanar, ghiggini2021surface}.

\begin{prop}  
Let $(Y,\xi)$ be a contact 3-manifold admitting a planar multilink open book. Then any strong symplectic filling of $(Y, \xi)$ is negative definite, cannot contain symplectic surfaces of positive genus, and its intersection form embeds in the standard diagonal lattice. 
\end{prop}

The topological constraints imply that for a surface singularity whose link admits a 
planar multilink open book, the resolution graph must be a tree of spheres, and that its intersection lattice admits a certain restrictive type of embedding into the diagonal lattice (see Definition~\ref{def:sandwichchar}). We then adapt combinatorial methods of~\cite{stipsicz2008rational}  to prove 
Theorem~\ref{thm:equivalence}, by giving a lattice-theoretic characterization of sandwiched graphs in Theorem~\ref{thm:sandwich}.

To prove Theorem~\ref{thm:rational}, we first establish a relevant property of the Heegaard-Floer contact invariant \cite{OSCont}, again generalizing a known result on planar contact 3-manifolds \cite{OSSPlanar}.
\begin{theorem}\label{thm:contactvanishing}
	Let $(Y,\xi)$ be a contact 3-manifold admitting a planar multilink open book. Then the image of the contact invariant $c^+(\xi)\in HF^+_{red}(-Y)$ vanishes.
\end{theorem}
Theorem~\ref{thm:rational} then follows from the fact that the above property of the contact invariant of the link characterizes rational singularities by \cite[Theorem 1.2]{Bodnar2021heegaard}. One could perhaps 
expect Theorem~\ref{thm:contactvanishing} to hold for all spinal planar 
open books, via potential extensions of Wendl's results \cite{wendl2013hierarchy}
to the spinal case and  the equivalence between embedded contact homology and Heegaard-Floer theories, \cite{colin2020equiv}. Our approach using Stein cobordisms to certain model open books only applies to planar multilinks. The proof may be interesting in its own right, as we describe planar multilink open books and their monodromies for links of certain classes of rational singularities, and build Stein cobordisms from arbitrary planar multilink open books to these singularity links.

We collect the necessary definitions and known facts on surface singularities and (multilink) open books in Section~\ref{sec:sandw}. Section~\ref{sec:spinal} contains the background on nearly Lefschetz fibrations and the first topological corollaries for our setting. In Section~\ref{sec:s-embed}, we prove~Theorem~\ref{thm:equivalence}. In Section~\ref{sec:cobordisms}, we prove Theorems~\ref{thm:contactvanishing} and~\ref{thm:rational}. Finally, in Section~\ref{sec:kulikov} we extend our results to give a
characterization of surface singularities admitting planar multilink open books in terms of their dual resolution graphs. 
\subsection*{Acknowledgments}
We thank  Marco Golla, László Koltai, Alexander Kubasch, Andr\'as N\'emethi, Agniva Roy, Laura Starkston, Andr\'as Stipsicz and Luya Wang  for helpful conversations. OP has been partially supported by the NSF grant DMS 2304080. 
MB is partially supported by the Doctoral Excellence Fellowship Programme (DCEP)
funded by the National Research Development and Innovation Fund of the
Ministry of Culture and Innovation and the Budapest University of
Technology and Economics, and by the ERC Advanced Grant KnotSurf4d.

\section{Singularities and multilink open books}\label{sec:sandw}
In this section we give the necessary background on surface singularities and related topological constructions. A useful general reference is~\cite{nemethi2022normal}.
Let $(X,o)$ be the germ of a normal singular point of an analytic surface $X\subset \C^N$.
 By the classical theorem of Milnor \cite{milnor1968singular}, the intersection $X\cap S^{2N-1}_\epsilon(o)$ is a smooth 3-manifold (called the \textit{link} of $(X,o)$). The link is endowed with a canonical contact structure $\xi_{can}$, given by the complex tangencies to the link in the ambient space.
For every normal surface singularity there exists a  \textit{resolution}, that 
is, a  smooth $\tilde{X}$ 
and a map $\pi:(\tilde X,E)\to(X,o)$ that restricts to a biholomorphism $\tilde X\setminus E\to X\setminus\{o\}$. We take a {\em good} resolution, so that $E= \pi^{-1}(o)$ is a collection of exceptional curves that are  
smooth surfaces transversely intersecting at double points only. Then, $\tilde X$ is topologically a plumbing of disk bundles over these surfaces. The intersection form of $\tilde X$ is negative definite. 
We encode this manifold using the \textit{dual resolution graph} $\Gamma$, where each vertex 
$v \in \Gamma$ corresponds to a component $E_v$ of $E$.  Two vertices $v, w$ are connected by an edge if the components $E_v$ and $E_w$ intersect. Every vertex $v$ is labeled by an integer equal to the self-intersection $E_v \cdot E_v$  of $E_v$. In our setting, all exceptional curves $E_v$ will be given by spheres, and we will often blur the notational distinction between vertices and the corresponding curves.
The link $Y$ is equipped with a {\em plumbing structure} induced by its 
presentation as the boundary of the plumbing of disk bundles,
\cite{Neu}. 
For each vertex $v \in \Gamma$, we have a piece of the plumbing structure given by an oriented $S^1$-bundle over the sphere $E_v$ with holes (the number of holes corresponds to  the valency of the vertex), with the Euler number $E_v \cdot E_v$.  To form $Y$, these pieces are glued along their boundary tori, 
as dictated by $\Gamma$.

The canonical contact structure depends on the topology of the link only, rather than the analytic structure of 
$(X, o)$, and can be read off the resolution graph:  
\begin{theorem}[\cite{CNPP, ParkStipsicz}]
	Let $Y,Y'$ be links of the singularities $(X,o), (X',o')$ respectively. If $Y$ is diffeomorphic to $Y'$, then $(Y,\xi_{can})$ is contactomorphic to $(Y',\xi_{can})$. Moreover, $\xi_{can}$ can be described as the contact structure on the boundary of the plumbing of the symplectic disk bundles given by the resolution.  
\end{theorem}

 We will consider functions on the singular surface germ and the corresponding Milnor fibrations. Given a surface singularity germ 
 $(X,o)$ with link $Y$ and a holomorphic function $f\in \mathfrak m_{X,o}$, we have a 
 fibration $\arg f: Y \setminus L \to S^1$ in the complement of the vanishing locus $L= \{f=0\}\subset Y$. If $f$ defines an isolated singularity at $o$, the fibration gives a classical open book. In general, 
 we have a \textit{multilink} in the terminology of \cite{eisenbud1985three}.
 
 \begin{definition}
	A multilink open book for a 3-manifold $Y$ consists of an oriented link $L = \cup L_i \subset Y$ such that
	\begin{itemize}
		\item $Y\setminus\nu(L)$ fibers over $S^1$,
		\item the boundary of each page is the disjoint union of longitudes of the components $L_i$ of $L$, possibly with several copies of the longitude of the solid torus $\nu(L_i)$ for each $L_i$. 
	\end{itemize}
	We say that the binding component $L_i$ has multiplicity $m_i$ if there are $m_i$ disjoint copies of its longitude in the boundary of each page. 
	As in the classical case~\cite{giroux}, a contact structure $\xi= \ker \alpha$ is compatible with a multilink open book if $d \alpha$ is an area form on the pages, and the binding is positively transverse to $\xi$.  
\end{definition}

 A multilink open book that comes from a Milnor fibration associated to a function $f$ on the germ of a singularity is compatible with the canonical contact structure on $(Y, \xi_{can})$, \cite{CNPP}.
 The multilink open book given by 
$\arg f$  can be seen from the divisorial data of the function $f \in \mathfrak m_{X,o}$. We will specify the 
binding components (up to isotopy) and their multiplicities.  When $Y$ is a rational homology sphere, 
the binding determines the open book up to isotopy, \cite{CPP, PS2}. 

Consider the pullback $f\circ \pi$ to a good resolution $\tilde X\xrightarrow{\pi} X$, and assume additionally that the divisor of the total transform 
$\div (f \circ \pi)$ has normal crossings. Then $\div (f \circ \pi)$ 
decomposes as the strict transform $(f \circ \pi)_s$ and the exceptional part (a linear combination of the exceptional curves of the resolution, taken with some multiplicities). We have
\begin{equation}\label{eq:totaltransform}
\div f = (f \circ \pi)_e+ (f \circ \pi)_s = \sum r_v E_v+ \sum a_v A_v, \quad r_v, a_v \in \Z_{\geq 0}, 
\end{equation}
where $E_v$ denotes the exceptional curves of the resolution, and each $A_v$ is a 
non-compact divisor (also called a \textit{cut}) on $\tilde{X}$ such that $A_v \cdot E_v =1$ and $A_v \cdot E_w = 0$ for $w\neq v$.
We can picture the pair $(X,f)$ by extending the resolution graph with \textit{arrowhead} vertices corresponding to the strict transform of the function. To encode the divisor of the strict transform, the  
  arrowhead vertices come with multiplicities that indicate the vanishing order of the function. 

The binding of the multilink open book given by $\arg f$ can be described as the union 
of some $S^1$-fibers of the plumbing structure on $Y$. 
Specifically, a component of the strict transform 
$(f \circ \pi)_s$ that intersects $E_v$ gives a multilink 
binding component which is a fiber over $E_v$, with multiplicity that equals the multiplicity of the corresponding arrowhead: 
\begin{lemma}\label{lem:ob-from-f}   \cite{CNPP} A holomorphic function $f$ on $(X, o)$ whose total transform divisor is given by~\eqref{eq:totaltransform} induces a multilink open book decomposition of the link $(Y, \xi_{can})$ with the following binding data: each component $A_v$ with $A_v \cdot E_v=1$ contributes a binding component of multiplicity $a_v$, given by an $S^1$-fiber over the curve $E_v$ in the plumbing structure on $Y$.
\end{lemma}
In particular, the multiplicity  is $1$ if the corresponding component of 
$(f \circ \pi)_s$ is reduced, so that we get a classical open book if 
$(f \circ \pi)_s$ is a reduced divisor.

Because of the construction of open books from the strict transform of  holomorphic functions on $(X, o)$, we would like to know when a prescribed combination $\sum a_v A_v$ of arrowheads with nonnegative multiplicities can be realized as a strict transform of a function.  An equivalent question is when a prescribed effective  divisor $\sum r_v E_v$ can be the exceptional part of the total transform. Indeed,  
if $\div f= \sum r_v E_v+ \sum a_v A_v$ as in~\eqref{eq:totaltransform}, then
\begin{equation}\label{eq:div0}
0= (\div f) \cdot E_v= \left(\sum r_v E_v\right)\cdot E_v + a_v,
\end{equation} so the multiplicities of the arrowheads are uniquely determined by the exceptional part of the total transform of the function, 
and, conversely, $\sum a_v A_v$ uniquely determines $\sum r_v E_v$ since the intersection form 
$E_i \cdot E_j$ is negative definite \cite{grauert1962modifikationen}. 
By~\eqref{eq:div0}, we also see that the exceptional part $D_e= (f \circ \pi)_e = \sum r_v E_v$  must be in the Lipman cone
 \begin{equation}\label{eq:lipman}
 \mathcal{S}=\{D\in L\ |\ D = \sum r_v E_v, \quad  r_v > 0, \quad      \forall v\ D\cdot E_v\leq 0 \}.
 \end{equation} 
For rational singularities, the converse is also true: 

\begin{lemma}\cite[Remark 7.1.14]{nemethi2022normal}  If $(X, o)$ is a rational singularity, and $D\in \mathcal{S}$, then there is a holomorphic function 
$f$ on $(X, o)$ such that $(f \circ \pi)_e= D$.  
\end{lemma}
 In general, for any divisor $D\in\mathcal{S}$ there is a singularity $(X', o)$ 
 of the given topological type  such that $D$ is realized by a function on $(X', o)$ (\cite[Theorem 2.5.8]{nemethi2022normal}). We will not use this fact.

\begin{remark}
 A'Campo's formula computes the (topological) Euler characteristic of a smooth fiber of $f$ via the formula
\begin{equation}\label{eq:chi}
	\sum_{v\in\Gamma} r_v(2-d_v),
\end{equation}
 where the sum runs over vertices of $\Gamma$, and $d_v$ is the degree (valency)  of vertex $v$ including arrowhead vertices corresponding to components of $\div(f)$,  \cite[4.1.22]{nemethi2022normal}.  The number of boundary components corresponding to an arrowhead $a$ is given by $\gcd(r_v,a_v)$, where $v$ is the vertex adjacent to the arrowhead $a$, \cite[4.1.C]{nemethi2022normal}. This allows us to determine the genus of the smooth fiber. 
 
 The multilink open book from Lemma~\ref{lem:ob-from-f} can be converted into an honest open book if the strict transform of the holomorphic function is deformed to obtain a reduced divisor. 
 This results in cabling the multilink binding \cite{eisenbud1985three}, and by~\eqref{eq:chi} the new open book typically has higher genus.  
% by computing \eqref{eq:chi} and equating its value with $2-2g-\sum_\Gamma\gcd(r_v,a_v)$. 
\end{remark}

In this paper, we work with rational and sandwiched singularities. 
By definition, rational singularities have geometric genus  
$p_g= h^1(\mathcal O_{\tilde X})$ equal to zero. They can also be characterized by the fact that their links are L-spaces in the Heegaard Floer sense, \cite{nemethi2017links}.  Rationality can be detected combinatorially from the dual resolution graph (using Laufer's algorithm, which we describe and apply in Section~\ref{sec:cobordisms}).  Sandwiched singularities are a subclass of rational singularities, defined combinatorially as follows.
\begin{definition}[{\cite[Definition 1.9]{spivakovsky1987sandwiched}}]
	A dual resolution graph $\Gamma$ is called \textit{sandwiched} if it is a subgraph of another graph $\tilde\Gamma\supset\Gamma$, such that $\tilde\Gamma$ can be blown down to the empty graph.
\end{definition}
\begin{remark}\label{rmk:add-leaves} We can further assume that $\tilde\Gamma$ is 
obtained from $\Gamma$ by adding leaves with framing $(-1)$. 
This is equivalent to the above definition 
by~\cite[Proposition 1.13]{spivakovsky1987sandwiched}, see also Lemma~\ref{lem:spivakovsky}. Below, we will always assume that the augmented graph 
$\tilde\Gamma$ has this specific form.   
\end{remark}
\begin{example}
	The smallest example of a non-sandwiched graph is $D_4$. We can see this by trying to append $-1$ vertices to either one of the leaves or the node. In both cases, after blowing down twice we see a vertex with $0$ self-intersection, so the augmented graph is not negative definite, and it cannot blow down to the empty graph.
	\begin{figure}[ht!]
	\begin{minipage}{0.5\textwidth}
		\begin{tikzpicture}[scale=0.75, every node/.style={circle, fill=black, inner sep=1.1pt}]
			\node[label=above:$-2$] (c) at (0,0) {};
			\node[label=below:$-2$] (l1) at (0,-1) {};
			\node[label=below:$-2$] (l2) at (-1,1) {};
			\node[label=below:$-2$] (l3) at (1,1) {};
			\draw  (c)--(l1) (c)--(l2) (c)--(l3);
		\end{tikzpicture}\hspace{50pt}
	%\end{minipage}
		%\begin{minipage}{0.45\textwidth}
		\begin{tikzpicture}[scale=0.75, every node/.style={circle, fill=black, inner sep=1.1pt}]
			\node[label=above:$-2$] (c) at (0,0) {};
			\node[label=below:$-2$] (l1) at (0,-1) {};
			\node[label=below:$-3$] (l2) at (-1,1) {};
			\node[label=below:$-2$] (l3) at (1,1) {};
			\draw  (c)--(l1) (c)--(l2) (c)--(l3);
		\end{tikzpicture}
	\end{minipage}
	\caption{The graph $D_4$ (left) and a sandwiched surface singularity resolution graph.}
	\label{fig:sandwichex}
		\end{figure}
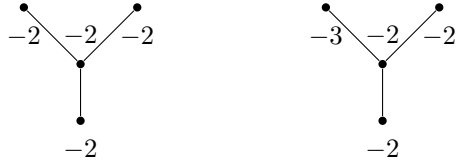
		On the other hand,  this graph becomes sandwiched if the framing on any vertex is lowered. The graph on the right of Figure~\ref{fig:sandwichex} blows down starting from either of the $-2$ leaves.
\end{example}

A sandwiched singularity can be encoded by a decorated plane curve singularity germ, \cite{djvs}. Appropriate deformations of the decorated germ lead to a description of deformations of the sandwiched  singularity and of symplectic fillings of its link, see \cite{djvs, PS2}. The plane curve 
germ  is constructed as follows: for each additional $(-1)$ curve augmenting the resolution as in Remark~\ref{rmk:add-leaves},  take (a germ of) a small complex disk $\ti{C}_i$  transverse to the $(-1)$ curve. When the augmented graph $\tilde\Gamma$ is blown down to a smooth point,  the image of these disks under the blowdown gives a (typically reducible) germ $C= \cup C_i$  of a curve singularity in $\C^2$. An integer decoration $w_i=w(C_i)$ for each irreducible component $C_i$ is given by the sum of the multiplicities of the corresponding disk on the exceptional divisors during the blowdown. One can reconstruct the dual resolution graph via the embedded resolution of the decorated germ (note that the decoration can encode additional blowups required when the strict transform of the germ has no more singularities), see 
\cite{djvs, PS2} for details.  We will use the following 
notation to record where the $(-1)$ vertices are attached to $\Gamma$ to 
form the
augmented graph $\tilde\Gamma$:  $E(C_i)$
denotes the exceptional curve (as well as the vertex of $\Gamma$) such that the disk $\ti{C}_i$ intersects a $(-1)$ curve attached to $E(C_i)$. Note that $E(C_i)$ may be the same for different $C_i$, if there are several $(-1)$ vertices  attached to the same 
vertex of $\Gamma$. Additionally, let $E_0$ correspond to the last surviving vertex of $\Gamma$ under the blowdown sequence for the graph~$\tilde\Gamma$.

For a sandwiched singularity $(X, o)$, 
\cite[Lemma 3.3]{PS2} constructs a planar multilink open book from the augmented 
sandwiched graph $\tilde\Gamma\supset \Gamma$ and the corresponding decorated germ $C= \cup C_i$. In more detail, this open book is obtained by 
constructing an appropriate holomorphic function on $(X,o)$ and examining its strict transform as in Lemma~\ref{lem:ob-from-f}.  Choosing coordinates on $\C^2$ so that none of the irreducible components 
$C_i$ of $C$ have vertical tangent cones, 
consider 
the projection $pr: \C^2 \to \C_x$ and take its pullback to the 
blowup of $\C^2$ 
that recovers the augmented graph $\tilde\Gamma$. Restricting to (a neighborhood of) $\Gamma$ and contracting the resolution graph, we get a holomorphic function $g: (X,o) \to \C$. By construction, 
the strict transform  $(g\circ \pi)_s$ of $g$ under the good resolution 
$\pi: \tilde{X} \to X$ given by $\Gamma$ can be computed from the multiplicities $m_i$ of the component $C_i$ at $0$: the function 
$g$ has an arrowhead of multiplicity $m_i$ at each $E(C_i)$, and additionally an arrowhead 
of multiplicity 1 at $E_0$.  (This follows from the calculation of multiplicities of the total transform of the projection function on the exceptional curves under the blowup of $\C^2$ according to $\tilde \Gamma$: 
the function $pr$ vanishes to order $m_i$ along the $(-1)$ curve attached to $E(C_i)$, see \cite[Lemma 3.3]{PS2}.)  

Another observation will be useful for the future: by tracking the multiplicities of the total 
transform of $pr$ inductively through each stage of the blowup, we see that the curve $E_0$ always enters with multiplicity $1$. Indeed, this is true at the first stage when we blow up a smooth point in $\C^2$ and get the strict transform of the smooth function $pr$ transverse to the exceptional curve. At the subsequent stages, the multiplicity of $E_0$ never increases in the total transform. It follows that the 
total transform of $g$ (with respect to the graph $\Gamma$) has the form 
\begin{equation}\label{eq:divisor}
 \div g = E_0 + \sum_{v \in \Gamma, E_v \neq E_0} r_v E_v + A_0+\sum_{i=1}^n m_i A_i,
\end{equation}
where, as in~\eqref{eq:totaltransform},  $\sum_{v \in \Gamma} r_v E_v$ is the exceptional part, and $A_0+\sum_{i=1}^n m_i A_i$ is the strict transform, constructed from the decorated germ by taking the arrowheads over the curves $E(C_i)$ with multiplicities $m_i$, and one more arrowhead $A_0$ of multiplicity $1$ over $E_0$.

We can now explicitly describe the multilink open book given by  Milnor fibration of $g$.

\begin{lemma} \cite{PS2} \label{lem:monodromygerm}
(1) The contact link $(Y, \xi_{can})$ of a sandwiched singularity
is supported by an open book with the binding given by the union of some
$S^1$-fibers (with multiplicities) of the plumbing structure given by the resolution graph $\Gamma$ on $Y$. The decorated germ $C= \cup C_i$
determines the binding data: each component $C_i$ of $C$ corresponds to a binding component of multiplicity $m_i$ given by a fiber over $E(C_i)$, and there is an additional binding component of multiplicity 1 given by a fiber 
over~$E_0$.  

(2) The multilink open book described in (1) is planar. The page is a disk with holes, where each arrowhead $A_i$ contributes $m_i$ holes, and the arrowhead $A_0$ corresponds to the outer boundary component of the disk. The monodromy is the same as the monodromy of its decorated curve germ (see below), up to the boundary parallel Dehn twists around the holes in the planar page. 
\end{lemma}

By the monodromy of the germ in part (2) we mean the monodromy of the surface bundle over $S^1$ that is obtained as the restriction of the projection $pr: \C^2 \to \C_x$ to $(\C^2 \setminus C) \cap 
\{|x|=\epsilon\}$  for a small $\epsilon >0$. The fiber of this bundle is a disk (a complex line) punctured at the intersection points with $C$, where each component $C_i$ contributes the number of punctures equal to its multiplicity $m_i$.  This fiber bundle is closely related to an open book supporting the link  $(Y, \xi_{can})$ of the sandwiched singularity $(X, 0)$. 
The relation comes  from de Jong--van Straten description of the Milnor fibers of $(X, 0)$ via certain deformations of the decorated germ $(C, w)$. After a deformation, the curve configuration only has transverse multiple points
as singularities. Blowing up at all these points and the additional free marked points, determined by  the decorations $w(C_i)$, one takes the complement of (the tubular neighborhoods of) the strict transforms of the components to obtain a Milnor fiber. A compatible multilink open book, closely related to the fiber bundle over $S^1$ given by the germ complement above, can be constructed on the boundary of the Milnor fiber. Indeed, if one ignores the 
boundary Dehn twists around the punctures in the planar page by taking the forgetful map to the braid group, the images of the two monodromies in the braid group coincide.

\begin{remark} \label{rmk:bdrytwists} The boundary Dehn twists are related to the blowups at the free marked points on the components, in turn determined by the decorations $w(C_i)$, so the complete monodromy can be recovered from the decorated germ $(C, w)$. In particular, taking sandwiched singularities with the same curve germ $C$ but increasing the decorations, one can add an arbitrarily high number of positive boundary Dehn twists around each hole of the page of the open book. Note that the same germ with the higher decorations corresponds to a singularity with a different dual resolution graph (the graph acquires chains of $(-2)$ vertices accompanying  the $(-1)$ vertex augmentations), which also changes the open book in part (1) of Lemma~\ref{lem:monodromygerm}. 
\end{remark}

\section{Planar nearly Lefschetz fibrations}\label{sec:spinal}

Fibered multilinks are a very special case of spinal open books, whose symplectic fillings were recently 
studied by Roy--Wang~\cite{HRW}. We will describe the structural results of~\cite{HRW} in our particular case of open books given by planar multilinks, but first we need to clarify the connection between multilinks and the more general spinal open books. We focus on the multilink case to summarize the main points  
of \cite{HRW} in the context relevant to us, referring the reader to the original paper for full details of the general definitions. 

The notion of a  spinal open book replaces a tubular neighborhood of a binding in a classical open book 
by a more general {\em spine}. The latter can be a circle bundle over an arbitrary surface. 
Multilink open books satisfy the hypothesis of being {\em uniform with respect to the disk}, which plays an important role in~\cite{HRW} and imposes further restrictions on the spine. 

In full generality, a spinal open book uniform with respect to $D^2$  has a 
spine given by an $S^1$-bundle over a {\em vertebra} surface $V$, such that  $V$ is a simple branched cover of the disk $D^2$, see \cite[Definition 2.3]{HRW}. For example, $V$ can be an annulus, so that the pages of the open book are attached along the boundary to the two tori that form the boundary of the  spine $V\times S^1$. For multilink open books, each connected component of the spine is a solid torus, so that $V=D^2$ for each component, but the simple branched covering $V \to D^2$ given by the cross-section of a spine component may be non-trivial: the number of sheets equals the multiplicity of the corresponding component of the binding. It is important to note that due to the complexity of the spine, the (spinal) monodromy does {\em not} fully determine a general uniform spinal open book, unlike the classical case. However, a multilink open book (and the contact 3-manifold that it supports) is determined by the monodromy $\Phi$ of the surface bundle over $S^1$, given by the 
complement of the fibered multilink. 

We describe this construction in some detail. 
Consider $\R^3 =\{(r, \theta, z )\}$
with cylindrical coordinates and 
its standard contact structure $\ker({dz +r^2 \,d\theta})$. In the complement of the $z$-axis, the fibration over $S^1$ given by the vertical planes $\theta= const$ gives a model for a neighborhood of the binding of a classical open book. For a multilink open book with a binding component of multiplicity $m$, the model is given by the pullback of this open book 
via the branched covering map $(z, w) \mapsto (z, w^m)$, where $w=r e^{i\theta}$. 

Each page meets
this spine component along $m$ vertical planes in this model, cutting $m$ 
parallel longitudes on its boundary torus. If we label the corresponding boundary components of the page by $1, \dots, m$, in the cyclic 
order given by their $\theta$-coordinate, then the monodromy of the open book permutes these page boundary components by the cyclic permutation $(12 \dots m)$. For a  multilink open book with page $P$ 
and monodromy $\phi$, with $k$  binding components of multiplicities $m_1, \dots, m_k$, the monodromy permutation $\sigma$ of the boundary components of the page is the product  of $k$ cycles of length $m_1, \dots, m_k$. 

The 3-manifold with its multilink structure, up to isotopy, can then be reconstructed from 
$(P, \phi)$, by gluing in $k$ solid tori (neighborhoods of the binding components) to the boundary of the mapping torus of $\phi$ as prescribed by the permutation $\sigma$ and the standard model of the binding component of the given multiplicity. Assuming that the boundary components are labeled by numbers $1, \dots, m$ (where $m=\sum_{i=1}^k m_i$), we can say that $\phi$ induces the permutation 
$\sigma \in S_m$, and there is an exact sequence
$$
0 \to Mod (P, \d P) \to SMod(P) \to S_m \to 0.
$$
Here,   $SMod(P)$ is the spinal mapping class group, see \cite[Section 2.5]{HRW}.   Elements of $SMod(P)$  are orientation-preserving diffeomorphisms (up to isotopy) that may permute the boundary components of $P$ according to a specific model.
The classical mapping class group $Mod (P, \d P)$ consists of the isotopy classes of orientation-preserving diffeomorphisms of $P$ fixing $\d P$ pointwise.  

 It is easy to see that every multilink open book fits the definition of a spinal open book uniform with respect to the disk, 
\cite[Definition 2.3]{HRW}: a simple branched covering on the vertebrae postulated by \cite{HRW} can be obtained by deforming the standard cyclic covering described above. Conversely, every spinal open book which is uniform with respect to the disk and has disk vertebrae is a multilink open book, since disk vertebrae force all spine components to be solid tori isomorphic to a neighborhood of a binding component of multiplicity $m$.

The results of \cite{HRW} say that for planar spinal contact 3-manifolds uniform with respect to the disk, every strong symplectic filling is given by a {\em nearly Lefschetz} fibration over the disk compatible with the given spinal 
open book on its boundary (in particular, the fiber is the planar page of the spinal open book).  A nearly Lefschetz fibration is a map $\pi:W\to D^2$ that has finitely
many singular fibers of two types. The first type is Lefschetz singularities, locally modeled on $(x,y)\mapsto x^2+y^2$. The second type is ``exotic'' singular fibers locally modeled on the complement of the tubular neighborhood of the multiple-valued section
$\{x=y^2\}$ of $\C^2 \to \C$ near its branch point, that is, the model for the fibration is given by the projection  $\C^2 \setminus \{ y^2=x\} \to \C$ near 0, \cite[Definition 2.9]{HRW}.
At these ``exotic'' singularities, two boundary components of the page merge into one. A key example of a nearly Lefschetz fibration is given by the complement of a {\em multisection} in a classical Lefschetz fibration over the disk; by definition, a multisection is a simple branched covering of the base. It is not hard to see that every nearly Lefschetz fibration in fact arises as the complement of a (possibly disconnected) multisection in an appropriate Lefschetz fibration, see \cite{HRW, BH} for details.

A spinal open book structure is induced on the boundary of a 4-manifold equipped with a  nearly Lefschetz fibration: the mapping torus part of the spinal open book is given by the vertical boundary of the fibration, and the spine is given by its horizontal boundary. When the nearly Lefschetz fibration is given as the complement of a multisection, it is easy to see that the vertebrae for the spine are given by the connected components of the multisection (the multisection gives a branched covering of the disk by the vertebrae, so that the spinal open book is uniform with respect to the disk). In particular, when the multisection components are all homeomorphic to disks, we get a multilink open book.   

Similarly to the classical case, nearly Lefschetz fibrations compatible with the given spinal open book on the boundary can be encoded by positive monodromy factorizations. The factorizations are considered in the spinal mapping class group 
$SMod(P)$. A nearly Lefschetz fibration gives a factorization where each Lefschetz singularity contributes a positive Dehn twist around the corresponding vanishing cycle, and each ``exotic'' singularity contributes a {\em boundary interchange} acting as a transposition on the two boundary components of the page merged at the exotic fiber. The boundary interchange is a half-twist along an arc in the page $P$ connecting these two components of $\d P$,  
see \cite[Figure 1]{HRW} for a specific model. For a spinal open book uniform with respect to $D^2$, each boundary interchange contributes a branch point of the covering of the disk by the vertebrae: the topology of the vertebra is determined by the branched covering via the Riemann--Hurwitz formula. 

Conversely, a nearly Lefschetz fibration filling of a spinal open book
can be constructed from a factorization of the spinal monodromy.
Since the topology of the spine must be preserved, one is forced to consider only {\em admissible} positive factorizations with respect to the given spinal open book. Admissibility means that the combinatorics of the boundary  interchanges, although not the interchange arcs, needs to be fixed for all factorizations: if a particular spine component has {$m_i$} 
branch points for its vertebrae covering of the disk, then there are {$m_i$} boundary interchanges (along some arcs) of pairs of boundary components of the page that are attached to this spine component.

Since we require that each spine verterbra be a disk for multilink open books, each admissible factorization must have at most one boundary interchange connecting a pair of boundary components. For the permutation
$\sigma_\phi$ of the boundary components induced by the monodromy $\phi$, the cycles in the decomposition of $\sigma_\phi$ correspond to binding components as above. An admissible factorization decomposes each cycle of length $m_i$ into $m_i-1$ transpositions that correspond to the boundary interchanges associated with the given spine/binding component.

The main results of~\cite{HRW} state that {\em all} minimal strong symplectic fillings of a planar spinal contact 3-manifold arise as 
compatible nearly Lefschetz fibrations with the same fiber $P$.
Accordingly, if the open book is uniform with respect to the disk, all fillings can be encoded by positive admissible factorizations. We give a precise statement only for our specific case of multilinks. As discussed above, the corresponding open books are uniform with respect to the disk, and all vertebrae are disks, each giving a branched covering of the base disk whose degree is the multiplicity of the corresponding binding component. The page $P$ of the open book is a sphere with holes (whose boundary components can be permuted by the spinal monodromy), and the base of the nearly Lefschetz fibration is a disk, which leads to a very simple topological description.

\begin{theorem}[\cite{HRW}]\label{thm:nearlyL-sphere}
Let $(Y, \xi)$ be a contact 3-manifold supported by a planar multilink open book $(P, \phi)$. 
Then each minimal strong symplectic filling arises as the complement of a multisection in a Lefschetz fibration
$S^2 \times D^2 \# n \cptwobar  \to D^2$ for some $n\geq 0$. Connected components of 
the multisection are in bijective correspondence with the multilink binding components, with matching degrees/multiplicities.  Each connected component of the multisection is a disk. 

Equivalently, each filling is encoded by an admissible positive factorization of the spinal monodromy $\phi \in SMod(P)$. Such a factorization is given by a product of positive Dehn twists along essential closed curves and the boundary interchanges along some arcs connecting the components of $\d P$, 
where any two components are connected by at most one interchange. 
The spinal monodromy $\phi$ induces a permutation $\sigma$ of the components of $\d P$, and each boundary interchange induces a transposition. The cyclic orbits of the components of $\d P$ under $\sigma$ correspond to the components of the multilink binding, and each cycle is decomposed into transpositions by the boundary interchanges.  
\end{theorem}

\begin{cor}\label{cor:nearlyL-disk} 
Suppose that the multilink open book has a binding component of multiplicity 1, that is, the spinal monodromy fixes the corresponding component of $\d P$ pointwise. In this case, each minimal strong symplectic filling arises as the complement of a multisection in a Lefschetz fibration
$D^2 \times D^2 \# n \cptwobar  \to D^2$ for some $n\geq 0$. 
\end{cor}

\begin{remark} For the nearly Lefschetz fibrations in the above statements, the vanishing cycles are essential curves in the fiber, although the multisections live in non-minimal Lefschetz fibrations on a blowup of $S^2 \times D^2$ (resp. $D^2 \times D^2$), where the vanishing cycles are contractible in the fiber.
\end{remark}

We can use the above description to read off the basic topological invariants of fillings, 
in particular their intersection form, from the nearly Lefschetz fibration structure. It will be convenient to consider two cases: the general 
case as in Theorem~\ref{thm:nearlyL-sphere} as well as a special case as in Corollary~\ref{cor:nearlyL-disk}.  

Consider a nearly Lefschetz fibration $\pi:W\to D^2$ with planar fiber $P$, such that the orientations of 
$W$, $ D^2$ and $P$ are compatible.
Let $D^{2}_e\subset D^2$ be a smaller disk whose preimage contains every exotic fiber but no singular fibers. Set $W_0=\pi^{-1}(D_e)$, then $W$ is obtained from $W_0$ by attaching 2-handles along the vanishing cycles. 

By Mayer--Vietoris,  $H_2(W_0)=0$. In presence of a multiplicity 1 component as in
Corollary~\ref{cor:nearlyL-disk}, $W_0$ is a complement of a multisection in $D^2 \times D^2 \to D^2$. Then $H_1(W_0)$ is generated by the meridians of the connected components of the multisection. We have
\begin{equation}\label{eq:orbits-equiv}
H_1(W_0) = H_1(P)/\sim,
\end{equation}
where $P$ is the planar fiber (a disk with holes), so that the generators of $H_1(P)$ are the boundaries of the holes, and the equivalence relation relates the holes that are in the same orbit under the monodromy permutation.  
Note that the outer boundary component of the disk is excluded from consideration. 
In this case, $H_1(W_0)$ is a free abelian group. 
In the general 
case, $W_0$ is a complement of a multisection in $S^2 \times D^2 \to D^2$. Then the  meridians $\mu_i$ of  the components of the multisection generate $H_1(W_0)$ as before, but these are subject to a relation $\sum m_i \mu_i=0$ because of the sphere fiber,
so $H_1(W_0)$ is no longer free abelian in general.  (Here, $m_i$ is the degree of the multisection component with meridian $\mu_i$.) 
Next, we use the long exact sequence of the pair $(W, W_0)$: 

\begin{equation}\label{eq:LES}
	0\to H_2(W)\xrightarrow{\iota} H_2(W, W_0)\xrightarrow{\partial_*} H_1(W_0)\to H_1(W)\to 0.
\end{equation}

Since $H_2(W, W_0)$ is freely generated by the vanishing cycles $\alpha_i$ of $W$,
we conclude that $H_2(W)$ is given by linear combinations of vanishing cycles that 
are nullhomologous in $H_1(W_0)$. An element  $[x]=H_2(W)$,  
$[x]=\sum a_i[\alpha_i]$ can be represented by an oriented embedded surface in $W$, 
following the construction of~\cite[Section 2]{ghiggini2021surface} with some adjustments. We sketch the idea here, referring the reader to \cite{ghiggini2021surface} for a detailed argument and pictures.

First, assume that $\sum a_i[\alpha_i]$ is nullhomologous in the sphere with holes  $P$ rather than in $W_0$. Fix a point at infinity away from the holes in the complement of the curves $\alpha_i$, and orient all of them so that each has winding number $-1$ around $\infty$. With this orientation, 
each $\alpha_i$ is oriented as the boundary of the region in $P$ disjoint from $\infty$ and has winding number $0$ or $1$ around each hole. If the linear combination  $\sum a_i[\alpha_i]$ were nullhomologous in the disk with holes away from $\infty$, then the total winding number $\sum a_i w_{\alpha_i}$ would be zero around each hole. Since the nullhomology 
only holds in $P$, we get that the winding number  $\sum a_i w_{\alpha_i}$ must give the same value around each hole. Set 
\begin{equation} \label{eq:winding}
 w = \sum a_i w_{\alpha_i} (\text{any hole}),
\end{equation}
and let $\beta$ be a small circle around $\infty$; in other words, $\beta$ is the boundary (with the reversed orientation) of a large disk $B$ that contains all the holes 
but not $\infty$. Then in $B$, the linear combination   $\sum a_i[\alpha_i]+ w [\beta]$ has winding number zero around each hole. We can form an embedded oriented surface in $W$, with boundary in the homology class 
$\sum a_i[\alpha_i]+ w [\beta]$, as follows. In distinct fibers, consider the   regions (with holes) 
cut out by $\alpha_i$ in the disk $B$ with holes (away from $\infty$), and take the union of these regions. 
Then close up the holes by tubes, via arranging 
holes into canceling pairs guided by their contributions into the total winding number. Finally, to get a closed surface, cap off the curve $\beta$ by the small disk in $P$ containing  $\infty$, and $\alpha_i$'s by the Lefschetz thimbles. 

More generally, when $[x]$ is represented by a linear combination
$\sum a_i[\alpha_i]$ nullhomologous in $H_1(W_0)= H_1(P)/\sim$,  recall that the boundaries of two holes in $P$ are in the same equivalence class if and only if they are both meridians of the same connected component of the multisection. Such meridians can be connected by tubes corresponding to the vanishing arcs; note that this respects the orientation of the boundaries of the holes. We incorporate these tubes into the construction of the closed surface for $[x]$.

Let $[x]=\sum a_i\alpha_i$, $[x']=\sum a_i'\alpha_i$ be two homology classes in $H_2(W)$ expressed as linear combinations of the vanishing cycles as above, and assume that all curves $\alpha_i$ intersect transversely. Then the embedded surfaces we construct for $[x]$ and $[x']$ intersect at the critical points corresponding to the vanishing cycles that enter non-trivially 
into both $[x]$ and $[x']$. (The caps are disjoint if the vanishing cycles are disjoint, even if the curves are homologous.) Additionally, the two surfaces intersect at points that correspond to the intersections of the curves $\alpha_i$, $\alpha_j$ in $P$, but these contributions add up to zero since intersection points of two simple closed curves on a planar surface come in canceling pairs.
Therefore, as in 
\cite[Proposition 2.1]{ghiggini2021surface}, 
we have 
\begin{equation}\label{eq:intersect}
[x] \cdot [x']=-\sum a_i a_i'  \quad \text{ for } [x]=\sum a_i\alpha_i, \  [x']=\sum a_i'\alpha_i.
\end{equation}
This gives 
\begin{prop} \label{prop:neg-def}
Every strong symplectic filling of a planar multilink contact 3-manifold is negative definite.
\end{prop}
The proposition generalizes a well-known property of planar contact 3-manifolds, first established in~\cite{etnyre2004planar}, see also~\cite{ozsvath2005planar}. These papers predate Wendl's work~\cite{We} describing fillings of planar contact $3$-manifolds as planar Lefschetz fibrations; given Wendl's result, the negative definite property is immediate from a direct calculation as in Proposition~\ref{prop:neg-def} and in~\cite[Proposition 2.1]{ghiggini2021surface}. 

Using Stipsicz's result~\cite{Stip2003geography} that $2 \chi + 3 \sigma$ is uniformly bounded below for all Stein fillings of a given contact 3-manifold, 
we recover the following corollary, 
established in a more general form in~\cite{lisi2021spine}:

\begin{cor} Given a planar multilink open book supporting a contact 3-manifold $(Y, \xi)$, all its Stein fillings have uniformly bounded Euler characteristic and signature.   
\end{cor}

\begin{example}\label{ex:models} Consider a family of
surface singularities $X_n$ with the dual resolution graphs as in Figure~\ref{fig:model}. 
These singularities are sandwiched, since each graph blows down after changing the arrowhead to a $(-1)$-vertex on the left of Figure~\ref{fig:model}. The corresponding decorated curve is given by the image of the transverse disk on the $(-1)$-curve after the blowdown, decorated by $2n+1$. (Recall that the decoration on each component is an integer given by the sum of the multiplicities of intersection of the curve with the exceptional divisors in the blowdown.) Equivalently, it is not hard to see, by constructing the embedded resolutions,  that this sandwiched family can be represented by the curve germs  $y^n+x^{n+1}=0$ (with decoration $2n+1$).  The multiplicity sequence of $y^n+x^{n+1}=0$ is $(n,1,1,\dots)$, so the Scott deformation is given by a transverse $n$-fold intersection with  $n+1$ additional free marked points on the curve to match the decoration. (See \cite[Proposition 4.1]{PS1} or \cite[Proposition 1.10]{djvs} for discussion of Scott deformation.) 

From the Scott deformation, we can describe a Stein filling 
using methods from \cite[4.5]{PS2}: the fiber is a disk with $n$ holes corresponding to the branches of the curve $y^n+x^{n+1}=0$, the boundary interchanges are determined by the projection of the deformed curve thought of as a multisection,  the $n$-fold transverse intersection gives a vanishing cycle $\alpha_0$
parallel to the outer boundary, and the marked points give vanishing cycles $\alpha_1,\dots,\alpha_{n+1}$ around the boundaries of the holes. Note that the latter are homologous 1-cycles in $W_0$, see~\eqref{eq:orbits-equiv}. Up to equivalence, we can put one vanishing cycle around each hole, with the exception of the last hole that gets two vanishing cycles $\alpha_n$ and $\alpha_{n+1}$. Any pair $\alpha_i, \alpha_j$, $i, j>0$ around two different holes can be connected by a tube in $W_0$ through the holes in the page.   The cycles in  $H_2(W)$ forming the resolution graph can now be seen explicitly on the page, using~\eqref{eq:intersect} to recover the graph: the $(-n-1)$ vertex is represented by $\alpha_0-\alpha_1-\dots-\alpha_{n}$,  the long arm vertices 
by $\alpha_{k}-\alpha_{k+1}$ for $k=1,\dots,n-1$, and the $(-2)$ vertex adjacent to 
the $(-n-1)$ vertex by $\alpha_{n}-\alpha_{n+1}$. 

	\begin{figure}[ht!]
		\begin{minipage}{0.32\textwidth}
			\begin{tikzpicture}[scale=0.75, every node/.style={circle, fill=black, inner sep=1.1pt}]
				
				\node[style={fill=white}] (a1) at (-0.7, 0.7) {};
				\node[label=right:$-n-1$, style={draw=black, fill=white}] (s2) at (0,-1) {};
				\node[label=above:$-2$] (l1) at (0,0) {};	
				\node[label=above:$-2$] (l2) at (1,0) {};
				\node[style={fill=white}] (d) at (1.5,0) {};
				\node[style={fill=white}] (dd) at (2.5,0) {};
				\node[label=above:$-2$] (e) at (3,0) {};
				\node[draw, draw=white, fill=white] (asd) at (2,0) {$\dots$};
				\draw  (s2)--(l1) (l1)--(l2)--(d) (dd)--(e);
				\draw[->] (l1)->(a1);
			\end{tikzpicture}
		\end{minipage}
		\begin{minipage}{0.32\textwidth}
			\begin{tikzpicture}[scale=1]
				\draw (0,0) circle (50pt);
				\draw[color=red] (0,0) circle (45pt);
				\draw[color=green] (-1,0)--(1,0)
				node[midway, color=black, fill=white] {$\dots$};
				\draw[fill=white] (-1,0) circle (7pt);
				\draw[color=red] (-1,0) circle (14pt);
				\draw[fill=white] (1,0) circle (7pt);
				\draw[color=red] (1,0) circle (10.5pt);
				\draw[color=red] (1,0) circle (14pt);
				
			\end{tikzpicture}
		\end{minipage}
		\caption{Left: the surface singularities $X_n$, where the long arm has $n-1\geq1$ vertices of self-intersection $-2$,  not counting the node. Right: a factorization of their filling, with the vanishing cycles in red and the vanishing arcs in green. Note that only the last inner boundary component is encircled by two parallel vanishing cycles.}
		\label{fig:model}
	\end{figure}
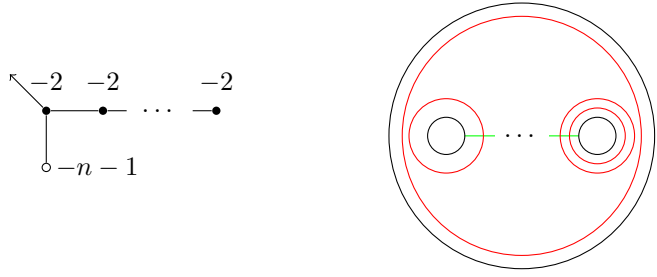
\end{example}

\begin{example}\label{ex:d4}
	There is a double covering $\partial A_3\to\partial D_4$ induced by the inclusion $\Z_4\leq Q_8\leq SO(4)$ compatible with the classical planar open book of $A_3$ and the Milnor fillable contact structures on the plumbing graphs as singularity links. This means that there is a planar spinal open book decomposition of the singularity link, but it does not admit a multiplicity 1 boundary component of the page; see Figure~\ref{fig:D4spinal} and cf.\ Remark~\ref{rmk:D4}.
	This example shows that the assumption on the multiplicity in Theorem~\ref{thm:equivalence} is necessary.
		\begin{figure}[ht!]
			\centering
		\begin{minipage}{0.32\textwidth}
			\begin{tikzpicture}[scale=0.75, every node/.style={circle, fill=black, inner sep=1.1pt}]
				\node[label=above:$(1)$, style={color=white, fill=white}] (a1) at (-2,2) {};
				\node[label=above:$(1)$] (v1) at (-1,2) {};
				\node[label=above:$(1)$] (v2) at (0,2) {};
				\node[label=above:$(1)$] (v3) at (1,2) {};
				\node[label=above:$(1)$, style={color=white, fill=white}] (a2) at (2,2) {};
				\draw  (v1)--(v2)--(v3);
				\draw[->] (v1)--(a1);
				\draw[->] (v3)--(a2);
				\node[label=above:$(2)$] (c) at (0,-2) {};
				\node[label=below:$(1)$] (l1) at (0,-3) {};
				\node[label=below:$(2)$] (l2) at (-1,-1) {};
				\node[label=below:$(2)$, style={color=white, fill=white}] (a3) at (-2,0) {};
				\node[label=below:$(1)$] (l3) at (1,-1) {};
				\draw  (c)--(l1) (c)--(l2) (c)--(l3);
				\draw[->] (l2)--(a3);
			\end{tikzpicture}
		\end{minipage}\centering
		\begin{minipage}{0.1\textwidth}
			\begin{tikzpicture}[scale=1]
				\draw (0,2) circle (25pt);
				\draw (0,2) circle (2.5pt);
				\draw[color=red] (0,2) circle (5pt);
				\draw[color=red] (0,2) circle (10pt);
				\draw[color=red] (0,2) circle (15pt);
				\draw[color=red] (0,2) circle (20pt);
				\draw[color=green] (0,-1)--(-0.88,-1);
				\draw (0,-1) circle (25pt);
				\draw[fill=white] (0,-1) circle (2.5pt);
				\draw[color=red] (0,-1) circle (5pt);
				\draw[color=red] (0,-1) circle (10pt);
				\draw[color=red] (0,-1) circle (15pt);
				\draw[color=red] (0,-1) circle (20pt);
			\end{tikzpicture}
		\end{minipage}
		\caption{The $A_3$ and $D_4$ graphs with the multiplicities and the  strict transforms of the functions inducing the open books indicated.}
		\label{fig:D4spinal}
		\end{figure}
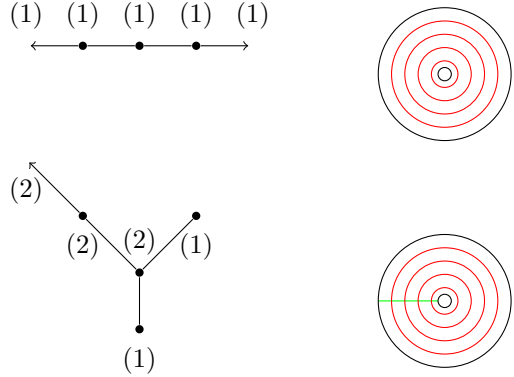
\end{example}

We need a generalization of another result  of~\cite{ghiggini2021surface}, namely the calculation of the first Chern class $c_1(J)$ for an almost complex structure $J$ compatible with the symplectic form $\omega$ on 
a symplectic filling $(W,\omega)$ of a planar multilink open book, given as a nearly Lefschetz fibration as above. We use the notation introduced earlier in this section: for a nearly Lefschetz fibration $W$ 
with the fiber $P$ and vanishing cycles $\alpha_i$, a homology class $[x]\in H_2(W)$ is given by a nullhomologous (in $W_0$) linear combination of vanishing cycles, $[x]= \sum a_i\alpha_i$. If the open book 
has a binding component of multiplicity 1, so that our nearly Lefschetz fibration is the complement of a multisection in a blowup of $D^2 \times D^2$, we have  $(c_1(J),x)=\sum a_i$ as 
in~\cite[Proposition 2.4]{ghiggini2021surface}, and the proof from~\cite{ghiggini2021surface} carries over verbatim. In general, for a nearly Lefschetz fibration whose fiber $P$ is a sphere with holes without a multiplicity 1 boundary component,  the evaluation $\langle c_1(J), [x]\rangle $ depends on the total winding number of the curves $\alpha_i$ around $\infty$, as defined earlier. Assuming that $\sum a_i[\alpha_i]=0$ in $H_1(W_0)$, the curves  $\alpha_i$ are oriented so that each has winding number $-1$ around $\infty$ and $\beta$ is a small circle around $\infty$, oriented to have  winding number $+1$, we write 
\begin{equation}\label{eq:x=sum}
[x]=w\beta+\sum a_i\alpha_i
\end{equation}
to encode the data for the class $[x]$. Equation~\eqref{eq:x=sum} can be viewed as notational shorthand (note that $[\beta]=0$ in $H_1(W_0)$), but it has a useful topological interpretation for the class $[x]\in H_2(W)$. 
Since $W$ is the complement of a multisection in the Lefschetz fibration on
 $S^2\times D^2\#n\overline{\C P^2}$, we can consider the class $i_*[x]$, where $i: W \to S^2\times D^2\#n\overline{\C P^2}$ is the inclusion. The homology $H_2(S^2\times D^2\#n\overline{\C P^2})$ is generated by the classes given by the regular fiber $S^2$ and the exceptional spheres $e_1, \dots, e_n$.  It is not hard to see that by construction, \eqref{eq:x=sum}~means that $i_*[x]$ is homologous to the class 
 $w[S^2]+ \sum a_i e_i$ in $H_2(S^2\times D^2\#n\overline{\C P^2})$. It is now easy to evaluate $c_1(J)$ on $[x]$ by naturality:
 $$
 \langle c_1(W), [x] \rangle=\langle c_1(S^2\times D^2\#n\overline{\C P^2}), i_*[x] \rangle=2w+\sum_i a_i.
 $$
This gives 
\begin{prop} 
		Let $(Y,\xi)$ be a contact 3-manifold admitting a planar multilink open book $(P,\phi)$.
	Let $(W,\omega)$ be a symplectic filling of $(Y,\xi)$ given by a nearly Lefschetz fibration with the fiber $P$ and oriented vanishing cycles $\alpha_i$.  Let $[x]= \sum a_i\alpha_i$ be given as a nullhomologous (in $W_0$) linear combination of vanishing cycles as above.
	Then for an almost complex structure $J$ compatible with $\omega$, 
	\begin{equation}\label{eq:c1} 
		\langle c_1(J), [x] \rangle =2w+\sum a_i \quad \text{ for } [x]= \sum a_i\alpha_i,
	\end{equation}
	where $w = \sum a_i w_{\alpha_i}$ is the total winding number of   $\sum a_i\alpha_i$ around each of the holes in $P$, defined as in~\eqref{eq:winding}. \qed
\end{prop}

%\begin{proof} This is exactly the same calculation as in, since in this case $W_0$ is the complement of a multisection in $D^2 \times D^2$, which allows to trivialize its tangent bundle and compute $c_1(J)$ as the obstruction to extending this trivialization over the 2-handles.   
%\end{proof}

\begin{example} Consider the nearly Lefschetz fibration for the resolution of the link of the $D_4$ singularity shown in~Figure~\ref{fig:D4spinal}. The vertices of the resolution graph are given as $-\alpha_1-\alpha_2,\ \alpha_1-\alpha_2,\ \alpha_2-\alpha_3,\ \alpha_3-\alpha_4$  using the boundary interchange with the outer boundary component. Using~\eqref{eq:c1}, we compute $c_1\equiv 0$, which is consistent with the vanishing of $c_1$ 
for the ADE singularities. Note that the contribution from the winding number is non-trivial for exactly one vertex. We will later see that this is true, in general, whenever we have a nearly Lefschetz fibration giving the minimal resolution of a link of surface singularity compatible with a planar multilink open book.
\end{example}

Together with adjunction, the formula for $c_1$ gives the following corollary.
\begin{cor}\label{cor:only-spheres}
	If a contact 3-manifold $(Y,\xi)$ admits a strong symplectic filling with a symplectic surface of positive genus, then it does not admit a planar multilink open book.
\end{cor}
\begin{proof}
Suppose that a nearly Lefschetz fibration filling $W$ as above 
contains a smoothly embedded symplectic surface in the homology class $[x] = \sum a_i\alpha_i$ as above, and 
$w$ is the winding number parameter, so that the surface is encoded by $w\beta+\sum a_i\alpha_i$ in the previous shorthand notation. Then the same surface gives  a symplectic class 
$\iota_*[x]= w[S^2]+ \sum a_i e_i$ in $H_2(S^2\times D^2\#n\overline{\C P^2})$, and since the fiber $S^2$ is symplectic, we conclude that $ w \geq 0$. From the adjunction formula $\langle c_1(\omega),[x] \rangle-[x]^2=2-2g$, we have 
	\begin{equation} \label{eq:adjunction}
		2 w + \sum a_i+\sum a_i^2=2-2g.
	\end{equation}
	Since $\sum a_i(1+a_i) \geq 0$, we must have $g=0$ if $w>0$. (This also forces $w=1$.) The case $w=0$ and $g=1$ is impossible since it would require $a_i=-1$ for all nonzero terms, and such a configuration of curves cannot be nullhomologous in $H_1(W_0)$ with total winding number $w=0$.
\end{proof}

We can get even more information on the representatives of the symplectic classes. 

\begin{cor} \label{cor:sympsphere}
	The only symplectic surfaces in a planar nearly Lefschetz fibration $W$ over $D^2$ given 
	by the complement of a multisection in a blowup of $S^2 \times D^2$ are 
	spheres. Their classes in $H_2(W)$ are expressed as nullhomologous linear combinations
	$$
	[x] = \alpha_j-\sum_I\alpha_i \text{ or } [x] = -\sum_I\alpha_i,
	$$
	 where at most one of the vanishing cycles enters with a $+$ sign, and the rest with a $-$ sign.
	 
	 In particular, these statements hold for every strong symplectic filling of a contact $3$-manifold admitting a planar multilink open book.  If there is a multilink component of multiplicity 1, then 
	 all symplectic sphere classes have the form $[x] = \alpha_j-\sum_I\alpha_i$.% (the second case is impossible).
\end{cor}

\begin{proof}
	By the previous corollary, every symplectic surface is a sphere. Examining the adjunction 
	formula~\eqref{eq:adjunction} more closely, we see that
	$$
	2 w + \sum a_i+\sum a_i^2=2
	$$
	can hold only in two cases. The first case is $w=0$ and one $a_i=1$, with the other nonzero coefficients equal to $-1$, as claimed. The second case is $w=1$ and  $a_i=-1$ for all nonzero terms. This second scenario does not arise when there is a multiplicity 1 multilink component: in this case we have 
	the complement of a multisection in a blowup of $D^2 \times D^2$, and, after an appropriate identification of the fiber, must have $w=0$, as discussed earlier in this section (in other words, a linear combination of vanishing cycles with all negative coefficients cannot be nullhomologous).  
\end{proof}

	%The intersection form of $W$ embeds into the standard negative definite lattice, with generators  
	%$\alpha_i$ corresponding to the vanishing cycles. Specifically, the classe

	%Their homology classes are described in terms of nullhomologous linear combinations of vanishing cycles of the form 
	%$\alpha_j-\sum_I\alpha_i$ or $-\sum_I\alpha_i$, These formulas give an embedding 
	%of the intersection form of $W$ 's.

 %These formulas give an embedding 	of the intersection form of $W$ into the standard negative definite lattice generated by $\alpha_i$'s.

%For planar multilink open books having a binding component of multiplicity 1, 
%and a symplectic filling $(W, \omega)$ with a nearly Lefschetz fibration,  
%we can also compute $c_1(J)$ for the almost complex structure compatible with $\omega$. 
%We can also see the evaluation of $c_1$ the same way as in \cite[Proposition2.4] with the observation that the almost complex line bundle is trivial over the preimage of the exotic fibers since it is the complement of a multisection of the trivial Lefschetz fibration $D^2\times D^2$.

%\section{Heegaard-Floer homology and the contact invariant}
%Lorem ipsum dolor sit amet. Don't need this as a separate section.

\section{Sandwiched graphs and $s$-embeddings} \label{sec:s-embed}

In this section, we prove Theorem~\ref{thm:equivalence}. When a contact link $(Y, \xi)$ is
compatible with a planar multilink open book with a multiplicity 1 binding component, its minimal resolution 
gives a strong symplectic filling represented by   a nearly 
Lefschetz fibration, with the basis for its homology lattice given by symplectic spheres. 
Then, Corollary~\ref{cor:sympsphere} establishes the existence of a 
particular embedding (called $s$-embedding below) of this homology lattice into the standard diagonal lattice. In particular, an $s$-embedding exists for a sandwiched graph, since the link of a sandwiched singularity has a required planar multilink open book, and the minimal resolution gives its strong symplectic filling.   Conversely, we can adapt some methods of \cite{stipsicz2008rational}   to prove that if a dual resolution 
graph of an arbitrary singularity admits an $s$-embedding, then
it must be sandwiched.

\begin{definition}\label{def:sandwichchar}
	We say that a negative definite plumbing graph $\Gamma$ admits an \textit{$s$-embedding} if it is a tree, and there is an embedding of quadratic lattices $\phi:(\Z^{|\Gamma|}\langle v_1,\dots,v_{|\Gamma|}\rangle,Q_\Gamma)\hookrightarrow(\Z^{N}\langle e_1,\dots,e_N\rangle,\langle-1\rangle^N)$ such that $\phi(\sum_{\Gamma'}v_i)=e_{\gamma'}-\sum_{j\in I_{\Gamma'}}e_j$ where $\gamma'\not\in I_{\Gamma'}$ for every connected subgraph $\Gamma'\leq\Gamma$.
\end{definition}

%\begin{remark}
%	Note that any connected subgraph of an $s$-embeddable graph is also $s$-embeddable.
%\end{remark}

\begin{remark}\label{rmk:sandwich-embed}
	Any sandwiched graph admits an $s$-embedding by Lemma~\ref{lem:monodromygerm} and Corollary~\ref{cor:sympsphere}. It is  easy to construct an $s$-embedding directly from the graph, without appealing to symplectic topology. We use induction: the single vertex $(-1)$ graph admits an $s$-embedding via the embedding $v\mapsto e$. Given a graph $\Gamma$ admitting an $s$-embedding  $\phi:Q_\Gamma\to\langle -1\rangle^N$, a blowup along any vertex or edge is also planar via the blowup formula. If we blow up at a vertex $v$, we construct an embedding into $(\Z\langle e_1,\dots,e_N,e_*\rangle,\langle-1\rangle^{N+1})$ by subtracting $e_*$ from the image of $v$, mapping the new $-1$ vertex to $e_*$, and keeping the image of every other vertex unchanged.
	Similarly if we blow up along an edge $(u,w)$, we construct the new embedding by 
	subtracting $e_*$ from the image of both $u$ and $w$ and mapping the new $-1$ vertex to $e_*$.
	Thus every graph that blows down to the empty graph admits an $s$-embedding, sandwiched graphs are by definition subgraphs of such graphs, and by definition any subgraph of an $s$-embeddable graph is also $s$-embeddable.
\end{remark}
We begin with two simple observations:
\begin{lemma}[{\cite[Lemma 3.6]{stipsicz2008rational}}]
	If $v,w\in\Gamma$ are two different vertices, then $e_v\neq e_w$.
\end{lemma}
\begin{proof}
	Suppose $e=e_v=e_w$, then $(v,w)=(e-\sum_V e_i,e-\sum_We_i)\leq -1$, a contradiction, since $(v,w)\in\{0,1\}$.
\end{proof}

\begin{lemma}\label{lem:nbrshared}
	If $v,w\in\Gamma$ are neighboring vertices with $\phi(v)=e_v-\sum_Ve_i,\ \phi(w)=e_w-\sum_We_i$ under the $s$-embedding $\phi$, we claim that $V\cap W=\emptyset$.
\end{lemma}
\begin{proof}
	The maximal intersection between vectors of this form is 2 if $w\in V$ and $v\in W$.
	In this case, $(v,w)=(\phi(v),\phi(w))=1$  implies that $V\cap W=\{j\}$, and the image sum of the connected subgraph $v$--$w$ contains $-2e_j$, a contradiction.
	If either $v\not\in W$ or $w\not\in V$, they cannot share negative basis elements since that would decrease the intersection of the embedding vectors to $\leq 0$.
\end{proof}

Since $\Gamma$ is a connected subtree of itself,
we see that $\sum_1^n\phi(v_i)=e_\gamma-\sum_{I_\Gamma}e_i$.
Expanding the terms of the sum we see that every $e_i$ can appear in the embeddings of the vertices 0, 1, 2 or 3  times with the following sign combinations:
\begin{enumerate}
	\item $+$
	\item $-$
	\item $+, -$
	\item $+, -, -$
\end{enumerate}
(We will refer to these as types 1--4. This corresponds to types 1,2,4,7 of \cite[Lemma 3.7]{stipsicz2008rational}). There is only one basis element of type 1, namely $e_\gamma$. For the rest, let us expand the sum $\sum_{i<j}(v_i,v_j)=|\Gamma|-1$. Since the $e_i$ are orthonormal, we can consider the basis element types separately. They contribute 0,0,1,1 respectively to the sum, 
i.e. there are $|\Gamma|-1$ basis elements of types 3 and 4.

Consider a type 2 basis element $e$. This means that there is a unique vertex $v\in\Gamma$ such that $\phi(v)=e_v-e-\sum_{V}e_i$ and $e$ does not appear in the image of any other vertex. Let us augment the graph $\Gamma$ by a leaf $l$ connected to $v$ and with framing $-1$, and extend the embedding by $\phi(l)=e$. 
\begin{prop}
	The extended graph $\tilde\Gamma$ constructed as above is negative definite and admits an $s$-embedding.
\end{prop}
\begin{proof}
	We need to check the subtrees which contain the new vertex $l$. Since it is a leaf, every connected subtree is obtained from a connected subtree of $\Gamma$ by the addition of $l$.
	For any such subtree $\Gamma'$ the corresponding basis element $e$ appears in its image if and only if $v\in\Gamma'$, since it is type 2. Adding the new vertex $l$ cancels out the basis element $e$ from the image, so it is still of the required form.
	
	For negative definiteness, we use induction on the number of vertices of $\Gamma$.
	The single vertex negative definite graphs admit $s$-embeddings, and they have a basis element of type 2 if and only if the vertex is framed $\leq -2$. The embedding is $e_1-e_2-\dots-e_n$, and by symmetry we can choose  $e_2$ as the image of the new $-1$ leaf. This lattice has Gram matrix $\begin{bmatrix}
		-n&1\\1&-1
	\end{bmatrix}$, which is negative definite if $n>1$.
	
	Consider now $\Gamma$ with $n$ vertices and a basis element of type 2. Since there is an embedding into a negative definite lattice, we know that the augmented graph $\tilde\Gamma$ must be negative semidefinite. If $F=\sum_{\tilde\Gamma} r_i v_i$ with $F^2=0$, then, expanding the $v_i$ under the embedding, we see that the coefficients of the $e_i$ have to sum to $0$ in each case. This means, in particular, that for the basis element of type 1 the coefficient is already 0.
	If we delete the vertex $v$ containing the type 1 basis element, the graph splits as a disjoint union of  graphs, each admitting an $s$-embedding, with one of these components augmented. By the inductive hypothesis,  these components are negative definite, thus the restriction of $F$ on them is zero, so  $F\equiv 0$.
\end{proof}
We may add new leaves for every basis element of type 2 
until we obtain a graph admitting an $s$-embedding with $|\Gamma|=N$. (By the discussion of the possible type combinations above, there must be an element of type 2 at each stage if the graph has fewer than $N$ vertices.) For this augmented graph, we know that there are no type 2 basis elements anymore; since we have $|\Gamma|=N$ number of vertices as well as basis elements; one of type 1, and $N-1$ of types 3 and 4. This new graph  has a lot of  $-1$ vertices; we will now blow them down, one by one. In the next lemma, we check that we never have to blow down a vertex of valency greater than 2, so the graph $\Gamma'$ is well defined.
\begin{prop}\label{prop:blowdown}
	Suppose that $\Gamma$ admits an $s$-embedding and $v\in\Gamma$ has $\phi(v)=e_v$. 
	Then blowing down the vertex $v$ gives a graph $\Gamma'$ that admits an $s$-embedding. If the embedding for $\Gamma$ has no type 2 vertices, then the same is true for $\Gamma'$.
\end{prop}
\begin{proof}
	$\Gamma'$ is connected and negative definite, since it is obtained via blowdown.
	The basis element $e_v$ may be of type 1, 3 or 4. In the first case the graph consists of a single vertex $v$ with framing $-1$, in the second $v$ is a $-1$ leaf, in the third $v$ has degree 2, and these are the only possibilities. This means that the blowdown results 
	in a graph $\Gamma'$, which is a tree. Algebraically, the blowdown is a projection to the orthogonal complement of $e_v$, so the subgraph sums also stay in the form required by the definition of a graph admitting an $s$-embedding. Similarly, this observation implies the statement about the type~2 vertices. 
\end{proof}
Let us blow down the augmented graph until we reach a minimal graph admitting an $s$-embedding. We  adapt another statement of Stipsicz-Szabó-Wahl to this simpler context and  provide the proof for completeness:
\begin{prop}[{\cite[Theorem 3.8]{stipsicz2008rational}}]\label{prop:minplanar}
	The only minimal graph admitting an $s$-embedding without basis elements of type 2 is the empty graph.
\end{prop}

\begin{proof}
	We proceed by induction: for single vertex graphs, only $-1$ has no type 2 basis element, but it is not minimal.
	Consider a minimal graph $\Gamma$ with $n>1$ vertices, admitting an $s$-embedding. Let $l\mapsto e_l-\sum_{I_L}e_i$ be a leaf such that $e_l$ is not of type 1 (such a leaf exists since there is only one basis element of type 1, and a tree with $n>1$ vertices has at least two leaves).
	We note that $e_l$ cannot be of type 4, since this would mean that in the expansion of $\sum_{v\neq l}v$ we get a term $-2e_l$.
	
	Thus either  $e_l$ is of type 2, and we are done, or  $e_l$ must be of type 3. Let $v$ be the neighbor of $l$. If $-e_l$ appears in the expansion of $v$, then contract the edge $(v,l)$, and give this new vertex framing $(v+l)^2=v^2+l^2+2\leq -2$. We obtain a minimal graph admitting an $s$-embedding on $n-1$ vertices by mapping this new vertex to $\phi(v)+\phi(l)$. This graph  has a basis element of type 2 by induction and the types of the basis elements besides $e_l$ are unchanged by Lemma~\ref{lem:nbrshared}, so $\Gamma$ also admits a type 2 basis element.
	
	If $-e_l$ appears in another vertex $w$, then remove $l$, increase the framing of $w$ by $1$ and correspondingly change $\phi(w)$ by adding $e_l$ to it. We once again obtain a graph admitting an $s$-embedding on $n-1$ vertices. If it is not minimal, then $\phi(w)=e_w-e_l$ in $\Gamma$. This means that $(l,w)\geq 1$, a contradiction. Thus the new graph is minimal. Then it has a type 2 basis element by the induction hypothesis, so $\Gamma$ also has a type 2 basis element.
\end{proof}

To summarize the above argument, we started with a graph admitting an $s$-embedding, augmented it with $(-1)$ vertices in a certain way, and showed that the resulting augmented graph blows down to the empty graph. This proves the ``only if'' part of the following theorem.   The ``if'' part was already established in Corollary~\ref{cor:sympsphere} and Remark~\ref{rmk:sandwich-embed}.
\begin{theorem}\label{thm:sandwich}
	A negative definite plumbing graph admits an $s$-embedding if and only if it is sandwiched.
\end{theorem}
Together with Corollary~\ref{cor:sympsphere}, Theorem~\ref{thm:sandwich} proves Theorem~\ref{thm:equivalence}.

\begin{example}
	The theorem requires the embedding assumption for all subgraphs, it is not enough to require the vertices to have images of the form $e_*-\sum_I e_i$, as shown by the embedding of a triangle singularity:
	$$
		\begin{minipage}{0.32\textwidth}
		\begin{tikzpicture}[scale=0.75, every node/.style={circle, fill=black, inner sep=1.1pt}]
			\node[label=above:$-1$] (c) at (0,0) {};
			\node[label=below:$-3$] (l1) at (0,-1) {};
			\node[label=below:$-4$] (l2) at (-1,1) {};
			\node[label=below:$-4$] (l3) at (1,1) {};
			\draw  (c)--(l1) (c)--(l2) (c)--(l3);
		\end{tikzpicture}
	\end{minipage}
	\begin{bmatrix}
	1 & -1 & -1 & -1 \\
	0 & 1 & -1 & -1 \\
	0 & 0 & 1 & -1 \\
	0 & 0 & -1 & 1 \\
	0 & -1 & 0 & 0 \\
	\end{bmatrix}
	$$
	A simple checking of cases shows that this graph is not sandwiched (in fact the singularity is not even rational, it is minimally elliptic).
\end{example}

\section{Stein cobordisms between multilink open books}\label{sec:cobordisms}

In this section we prove Theorem
~\ref{thm:contactvanishing} as a corollary of Theorem~\ref{thm:contactvanishingstatement} below.

\begin{proof}[Proof of Theorem~\ref{thm:contactvanishing}] Recall that by~\cite{nemethi2017links},
the links of rational singularities are L-spaces in Heegaard Floer homology, that is, 
they have $HF^+_{red}=0$. In this case the conclusion of  
Theorem~\ref{thm:contactvanishing} is trivially satisfied: 
the contact invariant $c^+(\xi_{can})$ is in the image of the $U$-map on $HF^+(-Y)$. Functorial properties  of  the contact invariant (\cite{OSCont}) then imply the theorem for any contact 3-manifold that admits a Stein cobordism to the contact 
link of a rational singularity. 
\end{proof}

\begin{theorem}\label{thm:contactvanishingstatement}
	Every contact 3-manifold $(Y,\xi)$ admitting a planar multilink open book admits a Stein cobordism $W:(Y,\xi)\to(Y',\xi_{can})$ to the link  $(Y',\xi_{can})$ of a rational singularity.
\end{theorem}

We will construct a family of certain  model rational singularities with explicit planar multilink open books, such that an arbitrary planar multilink open book can be transformed into one of the model open books by composing its monodromy with a number of positive Dehn twists around essential curves. Positive Dehn twists can be realized by Weinstein 2-handle attachments, producing a Stein cobordism required by the theorem. 
Planar multilink open books for sandwiched singularities play a key role in our proof. In this case, the page is a disk with punctures, the multiplicity of the outer boundary component is 1, and the monodromy of the multilink open book is given by the monodromy of an associated curve germ, 
plus an arbitrarily large number of positive boundary parallel Dehn twists around the holes in the page, see Lemma~\ref{lem:monodromygerm} and Remark~\ref{rmk:bdrytwists}. To construct a larger family of model multilinks without the multiplicity 1 restriction, 
we need the following lemma.

\begin{lemma}\label{lem:planarspinalgraph}
	Given a sandwiched presentation  $\Gamma\leq\tilde\Gamma$ for a sandwiched graph $\Gamma$, 
	there is another rational graph $\Gamma'$ with a planar multilink structure induced by a holomorphic function.  Topologically, the multilink open book for $\Gamma'$ is obtained by capping off the outer boundary component of the multilink open book page for $\Gamma$.% given by Lemma~\ref{lem:algplanar}.
\end{lemma}
\begin{proof}
For the link of the sandwiched singularity with the graph $\Gamma \subset \tilde\Gamma$, we take the multilink open  book from Lemma~\ref{lem:monodromygerm}, given by the holomorphic function $g$ described before the statement of that lemma. This open book is planar, its page is a disk with holes, whose outer boundary component has multiplicity $1$ in the binding and corresponds to the arrow $A_0$ of multiplicity $1$ over $E_0$ in the strict transform of $g$.  This boundary component will be capped off to create an open book compatible with the resolution graph $\Gamma'$. We construct $\Gamma'$ by increasing the framing of $E_0$ by $1$. Let  $D_e = \sum_{v \in \Gamma} r_v E_v$ be the exceptional part of the strict transform for $g$, and 
recall that by~\eqref{eq:divisor}, $r_0=1$.  We consider the divisor $D_e$ as an element in the intersection 
	lattice for $\Gamma'$ and check that $D_e$ is in the Lipman cone  $\mathcal{S}(\Gamma')$.
Indeed, for any $E_v\neq E_0$, clearly $E_v\cdot_{\Gamma'}D_e=E_v\cdot_{\Gamma}D_e\leq 0$ as before.	
At $E_0$, we have $E_{0}\cdot_{\Gamma'}D_e=E_{0}\cdot_{\Gamma}D_e+1$  since  $E_{0}\cdot_{\Gamma'} E_0 =E_{0}\cdot_{\Gamma}E_0 + 1$ and $r_0=1$, and 
$0= E_0 \cdot \div g = E_0 \cdot_{\Gamma} (D_e + A_0) =  E_{0}\cdot_{\Gamma} D_e+1$.

Next, we use Laufer's rationality criterion \cite[Theorem 4.2]{Laufer}
%\comm{add reference to Laufer}
to check that the graph $\Gamma'$ gives a rational singularity (see also Lemma~\ref{lemma:almostsandwich}).
%This criterion works for resolution graphs  where each exceptional curve has genus zero. %MB:if there is a higher genus vertex the graph cannot be rational anyway, the algorithm rejects it as well
Recall that the fundamental cycle $Z_{min}$ is, by definition, 
the unique coordinatewise minimal element of the Lipman cone $\mathcal S$. Since the divisor $D_e$ is in the Lipman cone both for $\Gamma$ and $\Gamma'$, and it has coefficient $1$ at $E_0$, we know that 
$Z_{min}(\Gamma)$ and  $Z_{min}(\Gamma')$ also have coefficient $1$ at $E_0$ (since every coefficient of $Z_{min}$ is positive). There is a ``computation sequence" for $Z_{min}$ of a given resolution graph, given as follows. Start with an arbitrary vertex, for example $Z_0= E_0$. At every step, if there is a vertex $E_v$ such that  $Z_k\cdot E_v>0$, then we
set $Z_{k+1} = Z_k+ E_v$. Once $Z_N \cdot E_v \leq 0$ for each vertex, we have found $Z_{min}=Z_N$. 
Laufer's criterion says that the singularity is rational if and only if at every step
where $Z_k \cdot E_v > 0$, we have $Z_k \cdot E_v =1$. We claim that if we start with $E_0$, the computation sequences for $\Gamma$ and $\Gamma'$ can be chosen to be the same. Indeed, since we know that $E_0$ enters in $Z_{min}$ with coefficient $1$ for both graphs, it means we never have to add $E_0$ at any step. Therefore the self-intersection $E_0 \cdot E_0$ never comes into play during the computation, and the other parameters of the two graphs are the same. It follows that the graph $\Gamma'$ is rational since $\Gamma$ is.   
  
It remains to relate the open books. Filling a boundary component of multiplicity $1$ corresponds to doing $0$-surgery on that binding component, with respect to the page framing. The outer boundary 
component of the open book for the sandwiched singularity encoded by $\Gamma$ is represented by a fiber over $E_0$ in the plumbing structure. We translate the page framing to the plumbing structure framing as follows. 
Notice that in a coordinate chart centered on the intersection point $E_v\cap A_v$, the function $g$ takes the form $x^{a_v}y^{r_v}$, where the exceptional divisor $E_v$ is locally $y=0$, and similarly $A_v$ is locally given by $x=0$.  In this system, the plumbing longitude is given by the tangents of the $y$, and the meridian by the $x$ coordinate.  The part of the page near $E_v$ is given by $P=\{x^{a_v}y^{r_v}=\epsilon\}$. 
Intersect the page with the torus generated by $x=\epsilon_1\exp(i\alpha),\ y=\epsilon_2\exp(i\beta)$. After adjusting the norms,  we get  the equation for the arguments as $a_v\alpha+r_v\beta=\arg\epsilon$. Tangent vectors to this level set have the form $-r_v\partial_\alpha+a_v\partial_\beta$. As mentioned before,  the plumbing meridian is generated by $\partial_\alpha$, and the longitude by $\partial_\beta$, so the page-framed $0$-slope transforms into the rational $-\frac{r_v}{a_v}$ slope. This means that the $0$ surgery with respect 
to the page framing is realized by   the addition of a new $(-1)$ leaf  to $E_0$ in the plumbing graph. Blowing down this $(-1)$ vertex,  we get the plumbing $\Gamma'$,  with the new open book binding given by  the collection of $S^1$ fibers encoded by the divisor $D$. \end{proof}

%PLEASE EXPLAIN THE FRAMING CALCULATION.
 %\blue{We extend the function to the manifold after the surgery, and see that the page-framed 0-surgery corresponds to $-\frac ba$ surgery on a fiber over $E_{v_0}$ where $b=m_{v_0}$ is the multiplicity of $D_e$ on the vertex and $a=-D_e\cdot E_{v_0}$ the multiplicity of the binding component.} Since the binding component has multiplicity $1$ and $E_0$ has coefficient $1$ in $D_e$, we get a $-1$ surgery relative to the fiber framing.  

\begin{remark}
	Note that the same capping-off procedure may be done with any other boundary component of a page of an open book constructed from an augmented sandwiched graph. %, as long as that binding component 	has multiplicity $1$.
	Capping off amounts to removing the corresponding curve 
	component $C_i$ from the germ $C=\cup C_i$: the same computation shows that this is equivalent 
	to adding a $-1$ framed vertex to the support of the curve and blowing down. We get something novel in the case when we do this to the outer boundary component, since no component $C_i$ in the germ corresponds to the outer boundary.
\end{remark}
\begin{remark}\label{rmk:D4}
	We get a new perspective for the $D_n$ graphs from this lemma, cf. Example~\ref{ex:d4}. Consider the sandwiched graph on the left of Figure~\ref{fig:cusp-d4}, represented by a single cuspidal cubic with 6 markings. Filling the outer boundary component, we get a planar spinal presentation for $D_4$. More generally, a cuspidal cubic with $n+2$ markings gives a planar spinal presentation of $D_n$.
	
	\begin{figure}[ht!]
\begin{minipage}{0.32\textwidth}
	\begin{tikzpicture}[scale=0.75, every node/.style={circle, fill=black, inner sep=1.1pt}]
		\node[label=above:$(2)$] (c) at (0,-2) {};
		\node[label=below:$(1)$, label=left:$-3$] (l1) at (0,-3) {};
		\node[label=below:$(2)$] (l2) at (-1,-1) {};
		\node[label=below:$(1)$, style={color=white, fill=white}] (a2) at (1,-3) {};
		\node[label=below:$(2)$, style={color=white, fill=white}] (a3) at (-2,0) {};
		\node[label=below:$(1)$] (l3) at (1,-1) {};
		\draw  (c)--(l1) (c)--(l2) (c)--(l3);
		\draw[->] (l2)--(a3);
		\draw[->] (l1)--(a2);
		\draw (0,2) circle (60pt);
		\draw[color=green] (-1,2)--(1,2);
		\draw[fill=white] (-1,2) circle (10pt);
		\draw[fill=white] (1,2) circle (10pt);
		\draw[color=red] (0,2) circle (55pt);
		\draw[color=red] (-1,2) circle (15pt);
		\draw[color=red] (-1,2) circle (20pt);
		\draw[color=red] (1,2) circle (15pt);
		\draw[color=red] (1,2) circle (20pt);
	\end{tikzpicture}
\end{minipage}\centering
\begin{minipage}{0.32\textwidth}
	\begin{tikzpicture}[scale=0.75, every node/.style={circle, fill=black, inner sep=1.1pt}]
		\node[label=above:$(2)$] (c) at (0,-2) {};
		\node[label=below:$(1)$, label=left:$-3$] (l1) at (0,-3) {};
		\node[label=below:$(2)$] (l2) at (-1,-1) {};
		\node[label=below:$(1)$, label=above:$-1$] (a2) at (1,-3) {};
		\node[label=below:$(2)$, style={color=white, fill=white}] (a3) at (-2,0) {};
		\node[label=below:$(1)$] (l3) at (1,-1) {};
		\draw  (c)--(l1) (c)--(l2) (c)--(l3);
		\draw[->] (l2)--(a3);
		\draw (l1)--(a2);
		\draw[color=white] (0,2) circle (60pt);
		\draw[color=green] (-1,2)--(1,2);
		\draw[fill=white] (-1,2) circle (10pt);
		\draw[fill=white] (1,2) circle (10pt);
		\draw[color=red] (0,2) circle (55pt);
		\draw[color=red] (-1,2) circle (15pt);
		\draw[color=red] (-1,2) circle (20pt);
		\draw[color=red] (1,2) circle (15pt);
		\draw[color=red] (1,2) circle (20pt);
	\end{tikzpicture}
\end{minipage}\centering
\begin{minipage}{0.32\textwidth}
	\begin{tikzpicture}[scale=0.75, every node/.style={circle, fill=black, inner sep=1.1pt}]
		\node[label=above:$(2)$] (c) at (0,-2) {};
		\node[label=below:$(1)$] (l1) at (0,-3) {};
		\node[label=below:$(2)$] (l2) at (-1,-1) {};
		\node[label=below:$(2)$, style={color=white, fill=white}] (a3) at (-2,0) {};
		\node[label=below:$(1)$] (l3) at (1,-1) {};
		\draw  (c)--(l1) (c)--(l2) (c)--(l3);
		\draw[->] (l2)--(a3);
		\draw[color=white] (0,2) circle (60pt);
		\draw[color=green] (-1,2)--(1,2);
		\draw[fill=white] (-1,2) circle (10pt);
		\draw[fill=white] (1,2) circle (10pt);
		\draw[color=red] (-1,2) circle (15pt);
		\draw[color=red] (-1,2) circle (20pt);
		\draw[color=red] (1,2) circle (15pt);
		\draw[color=red] (1,2) circle (20pt);
	\end{tikzpicture}
\end{minipage}\centering
		\caption{A sandwiched graph and the corresponding planar multilink graph after capping off the outer boundary component. Unlabeled vertices are framed $-2$. By redrawing the diagram to take another boundary component to the outside, we recover Figure~\ref{fig:D4spinal} for $D_4$.}
		\label{fig:cusp-d4}
	\end{figure}
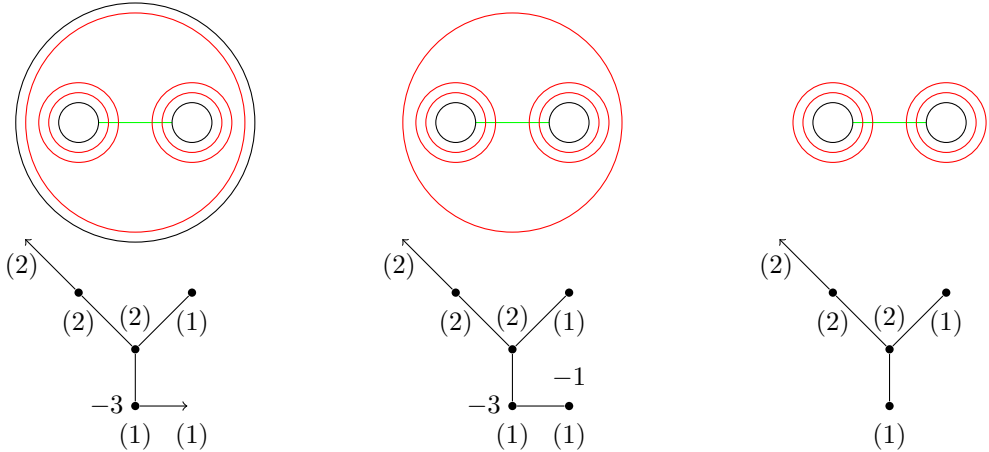
\end{remark}

Now we are ready to prove Theorem~\ref{thm:contactvanishingstatement}.

\begin{proof}[Proof of Theorem~\ref{thm:contactvanishingstatement}]
We first give a proof for the case of a planar multilink with a multiplicity 1 boundary component of the page $P$. In this case, we construct a cobordism to the link of a sandwiched singularity.  Represent the page as a disk with $m$ holes such that the outer boundary of the disk has multiplicity 1: this means that its monodromy $\Psi$ fixes the outer boundary of $P$ and induces a permutation of the holes. We can find a germ of a (reducible) singular plane curve $C$ such that the fiber bundle considered in Section~\ref{sec:sandw} also has an $m$-holed disk as its fiber, and 
whose monodromy induces the same permutation (under the appropriate labeling of the holes).
 To construct $C$,  decompose the permutation into cycles, so that $m= \sum m_i$, 
 where the $m_i$ are the lengths of the cycles. For a cycle of length $m_i$, 
 let the component $C_i$  of the germ be modeled on $y^{m_i} + x^{m_i+1}=0$: the monodromy of this germ gives a cyclic permutation of $m_i$ branches. The germ 
 $C = \cup C_i$ can be taken to be an arbitrary reducible singular germ whose irreducible components are as above. Working with the embedded resolution of $C$ and taking sufficiently large integers $w_i$ as decoration of the components, we can find a sandwiched singularity whose contact link has a multilink open book decomposition with monodromy $\Phi$ inducing the same permutation (see Lemma~\ref{lem:monodromygerm}). 

Further, for any $r>0$, we claim that 
there exists a sandwiched singularity whose contact link has a planar multilink open book with the monodromy 
\begin{equation}\label{eq:newphi}
\Phi' = \Phi \Delta^r \delta_1^r \dots \delta_m^r,
\end{equation}
where $\Delta$ stands for the positive Dehn twist around the outer boundary of the disk, and $\delta_1, \dots, \delta_m$ are the boundary parallel Dehn twists around the holes. Observe that for a curve germ $C$, there is another curve germ $C'$ such that $C$ is isomorphic to the strict transform of $C'$ after blowup. The monodromy of $C'$ is the same as the monodromy of $C$ composed with the positive Dehn 
twist $\Delta$ around all the punctures in the disk. Repeating this procedure $r$ times, we can find a curve germ whose monodromy equals
the monodromy of $C$ composed with $\Delta^r$ (and so the monodromy of the multilink open book for the associated sandwiched singularity  has the same property).   To add boundary twists around the holes in the open book, 
we increase the decorations on the curve germ. 

We wish to construct a Stein cobordism from the contact manifold supported by the given multilink open book $(P, \Psi)$ to the sandwiched singularity link with the open book $(P, \Phi')$ as in~\eqref{eq:newphi}, for $r$ large enough. 
It is a well-known fact that adding a Weinstein $2$-handle along an essential curve $\alpha$ on a page of an open book $(S, \phi)$ gives a Stein cobordism from the contact manifold represented by $(S, \phi)$ to that represented by $(S, \phi')$, where $\phi'$ is the composition of $\phi$ with the positive Dehn twist $D_{\alpha}$ around $\alpha$. Moreover, one can take a factorization of $\phi$ and insert $D_\alpha$ in an arbitrary place in the factorization to obtain a new contact manifold, since the mapping torus of the open book fibers over the circle (and only the cyclic order of the elements in the factorization matters).  The same is true in the multilink  case, since the Weinstein handles are attached away from the spine.  

Now, observe that because the permutations of the holes induced by $\Phi$ and $\Psi$  are the same, the diffeomorphism $\Psi \Phi^{-1}$ fixes the boundary $\d P$ of the $m$-holed disk $P$. As an element of the mapping class group of $P$, it decomposes as a product of (positive and negative) Dehn twists around essential curves in $P$. Iteratively using the lantern relation, we can further decompose it into a product of Dehn twists such that each of these either encloses two holes on the disk page or is parallel to the boundary of a single hole. We can introduce additional positive Dehn twists to kill all the negative Dehn twists in this factorization. Further, for each individual positive Dehn twist enclosing a pair of holes, we introduce a number of additional positive Dehn twists enclosing other pairs, guided by the daisy relation \cite[Figure 11]{PVHM} and \cite{HMVHM}, so that the product of the resulting collection is  given by several positive Dehn twists around the outer boundary, composed with a number of positive twists around the individual holes. Putting everything together, we have a recipe for creating a Stein cobordism from 
the open book $(P, \Psi)=(P, \Psi \Phi^{-1} \Phi)$ to an open book of the form $(P,  \Phi \Delta^r \delta_1^r \dots \delta_m^r)$, as required. The latter represents the link of a sandwiched singularity.

Now we address the general case, when there may be no component of multiplicity 1. Let $(S, \Psi)$ be such an open book, where $S$ stands for a {\em sphere} with $m$ holes. Perhaps after an isotopy, we may assume that $\Psi$ fixes a small disk in $S$. Cut out the interior of this disk to create a new open book $(P, \Psi)$, where $P$ is an $m$-holed {\em disk}, and we use the same notation for the restriction of the monodromy. By construction, the multilink open book $(P, \Psi)$ has a multiplicity 1 outer boundary component. The procedure in the first part of the proof gives a Stein cobordism to a multilink open book $(P, \Phi')$ that supports the contact link of a sandwiched singularity. Since this cobordism is given by the union of Weinstein 2-handles attached away from the binding, it will also induce a Stein cobordism from $(S, \Psi)$ to $(S, \Phi')$, when we cap off the outer boundary component $P$ to create multilink open books with page $S$.  By Lemma~\ref{lem:planarspinalgraph}, the open book   $(S, \Phi')$
represents the link of a rational singularity.  It is compatible with the canonical contact structure
since the open book comes from a holomorphic function.
\end{proof}
To conclude this section, we discuss the connection between planar multilinks and rational singularities. A more precise classification result will be established in Section~\ref{sec:kulikov}. 

\begin{proof}[Proof of Theorem~\ref{thm:rational}]

Let $(X, o)$ be a normal surface singularity such that $(Y, \xi_{can})$ is supported by a planar multilink open book. From Corollary~\ref{cor:only-spheres}, it follows that the dual resolution graph of $(X, o)$ must be a tree whose vertices are curves of genus $0$: higher genus exceptional curves have already been ruled out, and cycles in the graph can be used to create a symplectic torus when all the vertices are spheres. It follows that the link $Y$ is a rational homology sphere. By~\cite{Bodnar2021heegaard} the contact invariant $c^+(\xi_{can})$ is in the image of the $U$-map on $HF^+(-Y)$ if and only if the singularity is rational, so Theorem~\ref{thm:rational} follows from~Theorem~\ref{thm:contactvanishing}.
\end{proof}

\begin{remark}
Several other ideas can be used to demonstrate rationality of $(X, o)$, at least in some special cases.
If $(X, o)$ is smoothable, then any Milnor fiber $W$ has $b_2^+(W)=2 p_g (X, o)$, where $p_g$ is the geometric genus. By~Proposition~\ref{prop:neg-def}, it follows that $p_g=0$. More generally, if $(X, o)$ is not rational, then one can find non-negative definite symplectic fillings by using the known cases of the L-space conjecture, \cite{HRRW, SaRas}. Indeed, in the non-rational case the link $Y$ is not an L-space by~\cite{nemethi2017links}, and therefore it carries a taut foliation. We can then  construct a symplectic structure on $Y \times [-1, 1]$  and cap off $Y \times \{-1\}$ by a concave filling with $b_2^+>0$,  which results in a convex symplectic filling with $b_2^+>0$ for $Y \times \{+1\}$.  In the last sentence, we did not specify the contact structures, and one should check that the canonical contact structure can be obtained as perturbation of a taut foliation. (For Seifert fibered spaces over $S^2$, this can be achieved by working with  transverse taut foliations and transverse contact structures.)  
 \end{remark}

\section{Planar multilink singularities}\label{sec:kulikov}

In this section, we extend the preceding constructions and results on sandwiched singularities,
 cha\-rac\-te\-rizing normal surface singularity links that admit a planar multilink open book. We define the topological candidates.
\begin{definition}
	We call a resolution graph $\Gamma$ \textit{planar multilink} (or ${pm}$ for short) if 
	$\Gamma\leq\tilde\Gamma$, where $\tilde\Gamma$ is a negative semidefinite graph that blows down to a single $0$-framed vertex.
\end{definition}
We extend Spivakovsky's construction for sandwiched graphs  to give a more precise description of  $pm$ graphs.
\begin{lemma}[cf. {\cite[Proposition 1.13.]{spivakovsky1987sandwiched}}]\label{lem:spivakovsky}
	A resolution graph is planar multilink if and only if it can be augmented with $-1$ leaves, so that the resulting graph blows down to a single $0$ vertex.
\end{lemma}
\begin{proof}
	We provide some details of the proof for convenience.	
	Consider $\Gamma\leq\tilde\Gamma$ such that $\tilde\Gamma$ blows down to a $0$ vertex. Each blowup is either a vertex or an edge blowup. We modify the sequence of blowups as follows:
	if the blowup happens on a vertex in $\Gamma$,  keep it in the sequence; if both endpoints of an edge are in $\Gamma$, keep it in the sequence; if only one endpoint is in $\Gamma$, change it to a vertex blowup; otherwise leave it out. We obtain a possibly shorter blowup sequence, creating a graph $\tilde\Gamma'$ with a subgraph $\Gamma'$ which is graph-isomorphic to $\Gamma$, with possibly higher self-intersections  of some vertices. The self-intersections can be adjusted via further vertex blowups. This construction also guarantees that all vertices not in $\Gamma$ are $-1$ leaves, since we leave out any vertex that is not in $\Gamma$ during the blowup process.
\end{proof}
From now on, unless specified otherwise, we always assume that $\tilde\Gamma\setminus\Gamma$ consists of $(-1)$ leaves of $\tilde\Gamma$.
It is easy to see that {\em pm} graphs are closely related to sandwiched graphs:
\begin{lemma}[cf. Lemma~\ref{lem:planarspinalgraph}]\label{lemma:almostsandwich}
	Every planar multilink graph is obtained from an (augmented) sandwiched graph,  by increasing the framing of the last vertex to be blown down in a sandwiched graph by~$1$.
\end{lemma}
\begin{proof} Consider an augmented graph $\tilde{\Gamma}$ for $\Gamma$, blow it down to a $0$-framed vertex, then decrease the $0$-framing to $-1$ to blow down one more time. Reversing the blowdown sequence, we get a sandwiched graph related to $\Gamma$ as claimed. 
\end{proof}

 The main result of this section is
\begin{theorem}\label{thm:pm-graph}
 A normal surface singularity link admits a planar multilink open book decomposition if and only if its resolution is given by a  \textit{pm} graph.
\end{theorem}

The following lemma proves one direction of the classification and justifies our terminology: 
a $pm$ graph plumbing admits a planar multilink open book on its boundary, described algebraically. 

\begin{lemma}\label{lem:algplanar}
	Any  ${pm}$ graph $\Gamma$ with an augmentation $\Gamma\leq\tilde\Gamma$ admits an analytic realization, namely a function $f:\tilde\Gamma\to\C$ such that $\div(\pi\circ f)_e=\tilde F$ generates the null-space of $Q_{\tilde\Gamma}$. The restriction $F=\tilde F|_\Gamma$ gives a planar multilink for the graph $\Gamma$. In particular, every $pm$ graph admits a planar multilink open book compatible with the 
	canonical contact structure on the corresponding singularity link.
\end{lemma}
\begin{proof} This is similar to the constructions in Section~\ref{sec:spinal}:
	by following the blowup sequence, we construct the function $f$ by pullback, and keep track of its divisor $\tilde F$. Consider the trivial fibration $\C P^1\times D^2\xrightarrow f D^2$ with reduced structure. Let us blow up the central fiber $E_0$ according to the blowup sequence of $\tilde\Gamma$. The cycle $1E_0$ is the positive generator of the null-space of the single vertex graph with framing 0. By induction, consider a step in the blow-up sequence $\tilde\Gamma'$ and
	$\tilde F'$ on $\tilde\Gamma'$.
	We have two cases to consider. In the first case,  if the blowup happens at a vertex $v$ of $\tilde\Gamma'$, the new cycle will be $\tilde F'+m_ve_{new}$;  in the second case, if the blowup happens on an edge $(v,u)$ of $\tilde\Gamma'$, then the new cycle will be $\tilde F'+(m_v+m_u)e_{new}$, where in both cases $m_v$ is the coefficient of $v$ in $\tilde F'$, and $e_{new}$ is the new exceptional curve after the blowup.  The resulting cycle is the positive generator of the null-space, since its coefficients are relatively prime.
	
	By assumption, the last blowup step is a vertex blowup in each case, so after taking restriction $F=\tilde F|_\Gamma$, we see that if $F\cdot v=-m<0$, then the multiplicity of $v$ in $F$ is $m$. We constructed an analytic realization of $\tilde\Gamma$ with a given function $f$, where $\div(\pi\circ f)_e=\tilde F$.
	We saw that each component of $V(f)$ has the same multiplicity as the vertex it is supported on, so each component of $f|_{X_\Gamma}$ punctures each page $m$ times.

	We compute the Euler characteristic of the page by using~\eqref{eq:chi} for $\tilde\Gamma$ by induction (note that the answer only depends on the unlabeled graph and the cycle $\tilde F$). For the single vertex graph with $\tilde F=1v$, we get $\chi=2$. After blowing up a vertex $v$, 
	%\comm{which formula?}
	the Euler characteristic \eqref{eq:chi} decreases by $m_v$, since the degree of $v$  goes up by one, and we get a new term in the sum, namely an increase of $m_v$ from the new vertex, so the value is invariant. 
	
	For an edge blowup at $(u,v)$, the degrees of $u$ and $v$ stay unchanged, and we get a new vertex of degree $2$, which does not contribute to the sum, so it stays invariant with a value of $2$.
	
	For $\Gamma$, we remove leaves with multiplicity $m_i$, so the value of $\chi$ drops to $2-\sum m_i$. On the other hand, by the previous argument the fiber is an orientable surface with $\sum m_i$ holes, so it follows  that it is a disk.
\end{proof}

\begin{remark}
	This construction can be seen as a version of the Kulikov construction (see \cite[6.7.A]{nemethi2022normal}) for genus $0$, where we allow higher multiplicity components to be blown up as well. If we only allow blowups on multiplicity 1 components on a (trivial) genus $0$ family, we get the reduced fundamental cycle (i.e. minimal rational) graphs, which coincides with the set of graphs admitting a classical planar open book by \cite{ghiggini2021surface}.
\end{remark}

Lemma~\ref{lem:algplanar} admits an {\em algebraic}  converse:
\begin{prop}\label{prop:pmchar} If $(X, o)$ is a normal surface singularity with a resolution graph $\Gamma$, and $f:X \to\C$ a holomorphic function whose multilink Milnor fibration has genus $0$, then $\Gamma$ is a ${pm}$ graph.
\end{prop}

\begin{proof} Suppose that  $\tilde\Gamma$ is  a negative  semidefinite tree whose vertices are spheres. Recall that the Riemann--Roch function $\chi(D):=-\frac12(D,D+K)$ computes the holomorphic Euler characteristic, where $K=c_1(T X_{\tilde\Gamma})$ is the canonical class. We translate the requirement on the topological Euler characteristic into a holomorphic one. Let $\tilde F$ generate the null-space of $\tilde\Gamma$. 
 Expanding the identity  $0=\tilde F\cdot E$ where $\tilde F=\sum m_iv_i,\ E=\sum v_i$ for $v_i\in\tilde\Gamma$,  we get $0=\sum m_i(v_i^2+d_i)$, i.e. 
 $$
 \sum m_i(-d_i)=\sum m_iv_i^2.
 $$
 In the expression $\chi(\tilde F)=-\frac{1}{2}(\tilde F^2+K\cdot \tilde F)$, the first term is zero, since $\tilde F$ is in the kernel of the intersection map.  The second term expands, using adjunction, to $-\frac12\sum m_i(-v_i^2-2)$.
 Starting from the multilink open book planar equation $2=\sum m_i(2-d_i)$, substitute  to get $2=\sum m_i(2+v_i^2)$, and finally $\chi(\tilde F)=1$.
Now,  Zariski's lemma tells us that such a space $\tilde\Gamma$ is the central fiber of a family of genus 0 curves, and any such family is trivial (see \cite[6.7.28.]{nemethi2022normal}), that is, $\tilde\Gamma$ is a blowup on the central fiber of $S^2\times D^2$.
 
 Given a negative definite resolution graph $\Gamma$ of a singularity $(X, o)$ and a holomorphic function  $f:X_{\Gamma}\to \C$ with genus $0$ multilink Milnor fibration, 
 we would like to construct an augmented graph $\tilde{\Gamma}$ as above. 
 We can cap off  the boundary  components of the fiber of the Milnor fibration of $f$ 
via 0-surgery on its multilink components. A special case of this procedure was described in the proof of Lemma~\ref{lem:planarspinalgraph}. %Section~\ref{sec:kulikov}.
 The multilink components correspond to fibers over some spheres of its plumbing presentation (represented by the arrow components). 
The $0$-framed surgery with respect to the page converts to $-\frac{b}{a}$ surgery with respect to the plumbing framing, where $b$ is the multiplicity of $f$ on the component supporting the arrow, and $a$ is the multiplicity of the arrow. In general, 
this surgery is equivalent to  attaching linear graphs to $\Gamma$ at the arrows of $f$, corresponding to the Hirzebruch-Jung continued fraction coefficients of $-\frac ba$. Via this construction, every graph embeds into a negative semidefinite graph corresponding to a (possibly singular) family of genus $g$ curves, where $g$ is the genus of the Milnor fibration of $f$. If $g=0$, the family is trivial, as discussed previously. The divisor of $f$ extends to $\tilde F$, with $\chi(\tilde F)=1$ computing the genus of the family. 
\end{proof}

 The remaining part of Theorem~\ref{thm:pm-graph} is a {\em topological} converse of Lemma~\ref{lem:algplanar}. We will prove it as a corollary of Proposition~\ref{prop:p-embed} below and the results of Section~\ref{sec:s-embed}. Note that this argument is independent of  Proposition~\ref{prop:pmchar}.

\begin{prop} \label{prop:p-embed}
	$\Gamma$ is a ${pm}$ graph if and only if $Q_\Gamma$ embeds into a diagonal negative definite lattice, so that under the embedding the image of any connected subgraph is of the form $e_i-\sum e_j$ or $-\sum e_j$.
\end{prop}

\begin{proof}
	By analogy with Section~\ref{sec:s-embed}, we will call such embeddings \textit{p-embeddings}.
	If $\Gamma$ is a {\em pm} graph, take its planar multilink presentation and decrease the framing of the last vertex of the blowdown by 1 to get a sandwiched graph, as in Lemma~\ref{lemma:almostsandwich}. Then apply Theorem~\ref{thm:sandwich} to get an $s$-embedding of the new graph, and remove the type 1 basis element to go back to $\Gamma$.

	Conversely, given a $p$-embedding $\phi:Q_\Gamma\to\langle-1\rangle^N$, we wish to run the arguments of Section~\ref{sec:s-embed} again. To this end, we claim that there can be at most one vertex $w$ whose image is $\phi(w)=-\sum e_i$, and any connected subgraph $\Gamma'$ such that $w\in\Gamma'\leq\Gamma$ also 
	satisfies $\sum_{\Gamma'}\phi(v)=-\sum e_i$. 
	To see this, consider  $\Gamma'\leq\Gamma$ such that $\sum_{\Gamma'}\phi(v)=-\sum e_i$, and a vertex $v$ adjacent to this subgraph. From $(\Gamma',v)=1$ we get that $\phi(v)=e_v-\sum e_i$, such that $e_v$ appears in $\sum_{\Gamma'}\phi(v)=-\sum e_i$. Thus the larger connected subgraph $\Gamma'\cup\{v\}$ also sums to the form $-\sum e_i$, since $e_v$ cancels out.
	Thus if there is a vertex $v$ without a positive basis element in its image, it is unique: otherwise consider the path between two such to get a contradiction.  The whole $\Gamma$ also has this property. Adding a new positive basis element to the image of  $v$, we get a new graph that satisfies the requirements of Definition~\ref{def:sandwichchar}.
\end{proof}

\begin{proof}[Proof of Theorem~\ref{thm:pm-graph}] We need to prove that if the contact link of the singularity $(X, o)$ admits a planar multilink open book, then the dual resolution graph $\Gamma$ of $(X, o)$ is a {\em pm} graph. The strong symplectic filling given by the minimal resolution of $(X, o)$ has the structure of a nearly Lefschetz fibration compatible with the given planar multilink open book, that is, the filling is the complement of a multisection in a blowup of $S^2 \times D^2$.   By Corollary~\ref{cor:sympsphere}, the intersection lattice $Q_{\Gamma}$ of this filling admits a $p$-embedding. Then by Proposition~\ref{prop:p-embed}, $\Gamma$ is a {\em pm} graph.
\end{proof}

From Theorem~\ref{thm:pm-graph},  we get another proof of the rationality of $pm$ singularities via the Laufer sequence computation of Lemma~\ref{lem:planarspinalgraph}. In fact the model singularities in the proof of Theorem~\ref{thm:contactvanishingstatement}  encompass the entirety of planar multilink graphs!
\begin{cor}\label{prop:rational}
	Any ${pm}$ graph $\Gamma$ is rational.
\end{cor}

\begin{remark}
	Planar multilink graphs may be interpreted as "almost sandwiched" graphs analogously to almost rational graphs (\cite[Definition 7.3.21]{nemethi2022normal}), but the property is stronger here since we know that one vertex needs to be decreased by just $1$, not by an arbitrary amount. For example decreasing the short arm of $E_6$ to $-4$ we get a sandwiched graph, but $E_6$ is not planar multilink, since  increasing any vertex to $-1$ makes the intersection form indefinite. Alternatively by \cite{ko1939decomposition} we know that $E_6$ cannot embed into a diagonal lattice, but all planar multilink intersection forms do by Proposition~\ref{prop:p-embed}.
\end{remark}

%
%*************** DO NOT NEED THIS ************* (see Lemma 6.3) **************
%
%Note, that the function constructed in Lemma~\ref{lem:algplanar} always has a vertex $v$ where its multiplicity is $1$. If we modify $\Gamma$ by decreasing the self-intersection of $v$ by one, i.e. blow up $v$ at a smooth point in $\tilde\Gamma$ we get a planar multilink with a multiplicity 1 boundary component. By Theorem~\ref{thm:sandwich} this means, that the modified $\Gamma'$ is a sandwiched graph. Thus
%\begin{theorem}[cf. Lemma~\ref{lem:planarspinalgraph}]\label{cor:almostsandwich}
%	Every planar multilink graph is obtained by increasing the framing of the last vertex to be blown down in a sandwiched graph by 1.
%\end{theorem}
%
%

\bibliographystyle{alpha}
\bibliography{references}

\end{document}